\documentclass[12pt,reqno]{amsart}
\usepackage{latexsym}
\usepackage{amssymb}
\usepackage{mathrsfs}
\usepackage{amsmath}
\usepackage{fancybox,color}
\usepackage{enumerate}
\usepackage{enumitem}
\usepackage[latin1]{inputenc}
\usepackage{soul} 
\usepackage[colorlinks=true, linkcolor=magenta, citecolor=magenta]{hyperref}

\usepackage{color}

\usepackage{comment}

\numberwithin{equation}{section}

\makeatletter 
\renewcommand\subsection{\@startsection{subsection}{2}%
  \z@{-.5\linespacing\@plus-.7\linespacing}{.3\linespacing}%
  {\normalfont\bfseries}}
\makeatother

\usepackage[left=1.92cm,top=1.89cm,right=1.92cm]{geometry}

\def\1{\raisebox{2pt}{\rm{$\chi$}}}

\newtheorem{theorem}[equation]{Theorem}
\newtheorem{corollary}[equation]{Corollary}
\newtheorem{lemma}[equation]{Lemma}

\theoremstyle{definition}
\newtheorem{definition}[equation]{Definition}
\newtheorem{remark}[equation]{Remark}
\newtheorem{question}[equation]{Question}

\newcommand{\R}{{\mathbb R}}

\newcommand{\N}{{\mathbb N}}

\newcommand\diam{\operatorname{diam}}

\def\1{\raisebox{2pt}{\rm{$\chi$}}}

\newcommand{\Lip}{\operatorname{Lip}}

\def\vint_#1{\mathchoice%
        {\mathop{\kern 0.2em\vrule width 0.6em height 0.69678ex depth -0.58065ex
                \kern -0.8em \intop}\nolimits_{\kern -0.4em#1}}%
        {\mathop{\kern 0.1em\vrule width 0.5em height 0.69678ex depth -0.60387ex
                \kern -0.6em \intop}\nolimits_{#1}}%
        {\mathop{\kern 0.1em\vrule width 0.5em height 0.69678ex
            depth -0.60387ex
                \kern -0.6em \intop}\nolimits_{#1}}%
        {\mathop{\kern 0.1em\vrule width 0.5em height 0.69678ex depth -0.60387ex
                \kern -0.6em \intop}\nolimits_{#1}}}
\def\vintslides_#1{\mathchoice%
        {\mathop{\kern 0.1em\vrule width 0.5em height 0.697ex depth -0.581ex
                \kern -0.6em \intop}\nolimits_{\kern -0.4em#1}}%
        {\mathop{\kern 0.1em\vrule width 0.3em height 0.697ex depth -0.604ex
                \kern -0.4em \intop}\nolimits_{#1}}%
        {\mathop{\kern 0.1em\vrule width 0.3em height 0.697ex depth -0.604ex
                \kern -0.4em \intop}\nolimits_{#1}}%
        {\mathop{\kern 0.1em\vrule width 0.3em height 0.697ex depth -0.604ex
                \kern -0.4em \intop}\nolimits_{#1}}}

\newcommand{\aveint}[2]{\mathchoice%
        {\mathop{\kern 0.2em\vrule width 0.6em height 0.69678ex depth -0.58065ex
                \kern -0.8em \intop}\nolimits_{\kern -0.45em#1}^{#2}}%
        {\mathop{\kern 0.1em\vrule width 0.5em height 0.69678ex depth -0.60387ex
                \kern -0.6em \intop}\nolimits_{#1}^{#2}}%
        {\mathop{\kern 0.1em\vrule width 0.5em height 0.69678ex depth -0.60387ex
                \kern -0.6em \intop}\nolimits_{#1}^{#2}}%
        {\mathop{\kern 0.1em\vrule width 0.5em height 0.69678ex depth -0.60387ex
                \kern -0.6em \intop}\nolimits_{#1}^{#2}}}

\newcommand{\dist}{\operatorname{dist}}

\title[Maximal function estimates]{Maximal function estimates and   self-improvement   results for Poincar\'e inequalities} 

\thanks{The research is supported by the Academy of Finland.}

\author[J.\! Kinnunen]{Juha Kinnunen}   
\address[J.K.]{Department of Mathematics, Aalto University, P.O. Box 11100, FI-00076 Aalto University, Finland}
\email{juha.k.kinnunen@aalto.fi}

\author[J.\! Lehrb\"ack]{Juha Lehrb\"ack}   
\address[J.L.]{University of Jyvaskyla, Department of Mathematics and Statistics, P.O. Box 35, FI-40014 University of Jyvaskyla, Finland}
\email{juha.lehrback@jyu.fi}

\author[A.V.\! V\"ah\"akangas]{Antti V. V\"ah\"akangas}
\address[A.V.V.]{University of Jyvaskyla, Department of Mathematics and Statistics, P.O. Box 35, FI-40014 University of Jyvaskyla, Finland} 
\email{antti.vahakangas@iki.fi}

\author[X. Zhong]{Xiao Zhong}
\address[X.Z.]{Department of Mathematics and Statistics, Gustaf H\"allstr\"omin katu 2b, FI-00014 University of Helsinki, Finland}
\email{xiao.x.zhong@helsinki.fi}


\begin{document}

\begin{abstract}
Our main result is an estimate for   a   sharp maximal function, which implies a 
Keith--Zhong type   self-improvement property of  Poincar\'e inequalities   related to differentiable structures on metric measure spaces.
As an application, we give structure independent representation for Sobolev norms and universality results for Sobolev spaces.
\end{abstract}

\maketitle

\section{Introduction}

Relatively standard assumptions in analysis on metric measure spaces are a doubling condition on the measure and a Poincar\'e type inequality for a certain class of functions. 
Roughly speaking a Poincar\'e inequality transfers infinitesimal information encoded in the derivative to larger scales. It also relates the notion of a derivative to the given measure and, together with the doubling condition, implies Sobolev inequalities.
We consider   the so-called   $\mathcal D$-structures introduced in \cite{MR1876253}, which give a very general notion of a derivative with natural differentiation properties in metric measure spaces.
This gives an axiomatic point of view to the theory of Sobolev spaces on metric measure spaces, which includes the standard maximal and upper gradient approaches studied, for example, in \cite{MR2039955, MR1809341}.
Standard references to analysis on metric measure spaces are \cite{MR2867756,MR1800917,MR3363168}.

Keith and Zhong proved in \cite{MR2415381} that   Poincar\'e inequalities   are self-improving under certain assumptions.
More precisely,   their   result improves a $(1,p)$-Poincar\'e inequality with $p>1$ to a $(1,p-\varepsilon)$-Poincar\'e inequality for some $\varepsilon>0$.
This open ended property is of fundamental importance not only because of its theoretical interest but also because of its applications, for example, to regularity theory in the calculus of variations,   we refer to \cite{MR2415381} and references therein.  
In this work we establish   a corresponding self-improvement   property for $\mathcal D$-structures, see Theorem \ref{t.kz_D} below.
Our goal is to give an abstract and transparent argument with a special emphasis on the role of the underlying space and relevant maximal function inequalities.
Indeed,   instead of a good lambda inequality \cite[Proposition 3.1.1]{MR2415381},   our main result Theorem \ref{t.main_local} gives a new estimate for the sharp maximal function associated with a given $\mathcal D$-structure. 
This result may be of independent interest and several questions related to weighted norm inequalities for future research arise.

A distinctive feature of our approach is that,   in addition to   the standard Lipschitz scale, we   also   consider H\"older continuous functions.
  Moreover,   the role of the underlying space is visible only by way of the $\mathcal D$-structure   and certain geodesic arguments.  
On technical level our argument differs from that of \cite{MR2415381} in the sense that Whitney type extension theorems for Lipschitz functions are completely avoided and the stopping time argument is tailored for $\mathcal D$-stuctures. 
  We would also like to point out that there is only one single place in the proof 
of Theorem \ref{t.main_local}
where the assumed Poincar\'e inequality is needed.  
Another approach to the Keith--Zhong theorem has been recently given in \cite{EB2016}.

As an application of our main result we study universality results for Sobolev spaces related to $\mathcal D$-structures.
More precisely, Theorem \ref{t.structure} gives a   $\mathcal{D}$-structure   independent representation for the Sobolev norm.
We also show that any abstract Sobolev space, rising from a suitable $\mathcal{D}$-structure, is isomorphic to one particular Sobolev space. This extends and complements results in \cite{MR1809341,MR2508848}.

\section{Preliminaries}

\subsection{Tracking constants}\label{s.constants}

Our results are based on 
quantitative estimates and absorption arguments, where
it is often crucial to track the dependencies of constants quantitatively.
For this purpose, we will use the following notational convention:
$C({\ast,\dotsb,\ast})$ denotes a positive constant which quantitatively 
depends on the quantities indicated by the $\ast$'s but whose actual
value can change from one occurrence to another, even within a single line.

\subsection{Metric spaces}\label{s.metric}
Here, and throughout the paper, we assume that $X=(X,d,\mu)$ is a metric measure space equipped with a metric $d$ and a 
positive complete Borel
measure $\mu$ such that $0<\mu(B)<\infty$
for all balls $B\subset X$, each of which is always an open set of the form \[B=B(x,r)=\{y\in X\,:\, d(y,x)<r\}\] with $x\in X$ and $r>0$. 
As in \cite[p.~2]{MR2867756},
we extend $\mu$ as a Borel regular (outer) measure on $X$.
We remark that the space $X$ is separable
under these assumptions, see \cite[Proposition 1.6]{MR2867756}.
We  also assume that $\# X\ge 2$ and  
that the measure $\mu$ is {\em doubling}, that is,
there is a constant $c_\mu> 1$, called
the {\em doubling constant of $\mu$}, such that
\begin{equation}\label{e.doubling}
\mu(2B) \le c_\mu\, \mu(B)
\end{equation}
for all balls $B=B(x,r)$ in $X$. 
Here we use for $0<t<\infty$ the notation $tB=B(x,tr)$. 
In particular, for all balls $B=B(x,r)$ that are centered at $x\in A\subset X$ with radius
$r\le \mathrm{diam}(A)$, we have that
\begin{equation}\label{e.radius_measure}
\frac{\mu(B)}{\mu(A)}\ge 2^{-s}\bigg(\frac{r}{\mathrm{diam}(A)}\bigg)^s\,,
\end{equation}
where $s=\log_2 c_\mu>0$. We refer to \cite[p.~31]{MR1800917}.


\subsection{Geodesic spaces}
Let $X$ be a metric space satisfying the conditions stated in \S\ref{s.metric}.
By a {\em curve} we mean a nonconstant, rectifiable, continuous
mapping from a compact interval of $\R$ to $X$;  we tacitly assume
that all curves are parametrized by their arc-length.
We say that $X$ is a {\em geodesic space}, if 
every pair of points in $X$
can be joined by a curve whose length is equal to the distance between the two points. 
In particular, it easily follows that
\begin{equation}\label{e.diams}
0<\diam(2B)\le 4\diam(B)
\end{equation}
for all balls $B=B(x,r)$ in a geodesic space $X$.

The following lemma is \cite[Lemma 12.1.2]{MR3363168}.

\begin{lemma}\label{l.continuous}
Suppose that $X$ is a geodesic space and $A\subset X$ is a measurable set. 
Then the function
\[
r\mapsto \frac{\mu(B(x,r)\cap A)}{\mu(B(x,r))}\,:\, (0,\infty)\to \R
\]
is continuous whenever $x\in X$.
\end{lemma}


\begin{lemma}\label{l.ball_measures}
Suppose that $B=B(x,r)$ and $B'=B(x',r')$ are two balls in a geodesic space $X$ such
that $x'\in B$ and $0<r'\le \mathrm{diam}(B)$. Then
$\mu(B')\le c_\mu^3 \mu(B'\cap B)$.
\end{lemma}

\begin{proof}
It suffices to find $y\in X$ such that
$B(y,r'/4)\subset B'\cap B$.
Inequality $\mu(B')\le c_\mu^3 \mu(B'\cap B)$ then follows
from the doubling condition \eqref{e.doubling} and the fact that $B'\subset B(y,2r')$.

Assume first that $x\in B(x',r'/4)$. In this case we may choose $y=x'$, 
since we have for all $z\in B(x',r'/4)$ that
\[
d(z,x)\le d(z,x')+d(x',x)< r'/4+r'/4=r'/2\le \diam(B)/2\le r\,,
\]
and hence $B(x',r'/4)\subset B'\cap B(x,r)=B'\cap B$.

Let us then consider the case $x\not\in B(x',r'/4)$. Since $X$ is  a geodesic space, there exists an arc-length parametrized curve $\gamma\colon [0,\ell]\to X$ with
$\gamma(0)=x'$, $\gamma(\ell)=x$ and $\ell=d(x,x')$.
We claim that $y=\gamma(r'/4)$ satisfies the required condition $B(y,r'/4)\subset B'\cap B$.
In order to prove the claim, let us fix a point $z\in B(y,r'/4)$. Then
\begin{align*}
d(z,x')\le d(z,y)+d(y,x') <  r'/4 + d(\gamma(r'/4),\gamma(0))\le r'/2<r'\,. 
\end{align*}
Hence $z\in B(x',r')$ and therefore $B(y,r'/4)\subset B(x',r')=B'$. Moreover, since
$\ell=d(x,x')$, 
\begin{align*}
d(z,x)&\le d(z,y)+d(y,x)
<r'/4 + d(\gamma(r'/4),\gamma(\ell))\\
&\le r'/4+ (\ell - r'/4)=\ell=d(x,x')<r\,.
\end{align*}
It follows that $z\in B(x,r)=B$ and therefore $B(y,r'/4)\subset B'\cap B$.
\end{proof}

\subsection{H\"older and Lipschitz functions}
Let $A\subset X$.
We say that
$u\colon A\to \R$ is  a {\em $\beta$-H\"older function,} with 
an exponent $0<\beta\le 1$ and a constant 
$0\le \kappa <\infty$, if
\[
\lvert u(x)-u(y)\rvert\le \kappa\, d(x,y)^\beta\qquad \text{ for all } x,y\in A\,.
\]
If $u\colon A\to \R$ is a $\beta$-H\"older function, with a constant $\kappa$, then the classical McShane extension
\begin{equation}\label{McShane}
v(x)=\inf \{ u(y) + \kappa \,d(x,y)^\beta\,:\,y\in A\}\,,\qquad x\in X\,,
\end{equation}
defines a $\beta$-H\"older function $v\colon X\to \R$,
with the constant $\kappa$, 
which satisfies
$v|_A = u$; we refer to \cite[pp.~43--44]{MR1800917}.
The set of all $\beta$-H\"older functions $u\colon A\to\R$
is denoted by $\Lip_\beta(A)$. 
The $1$-H\"older functions are also called {\em Lipschitz functions}.
We denote $\Lip(A)=\Lip_1(A)$.

\section{Definition and basic properties of $\mathcal{D}$-structures}\label{s.abstract}

We adapt the terminology from \cite{MR1876253} 
concerning the so-called $\mathcal{D}$-structures.
This structural framework captures the properties
that we will need for Keith--Zhong type self-improvement
of Poincar\'e inequalities, treated in \S\ref{s.main}--\S\ref{s.axiomatic}.
In the following definition, and throughout the paper, we use the following familiar notation: \[
u_A=\vint_{A} u(y)\,d\mu(y)=\frac{1}{\mu(A)}\int_A u(y)\,d\mu(y)
\]
is the integral average of $u\in L^1(A)$ over a measurable set $A\subset X$
with $0<\mu(A)<\infty$. Moreover if $E\subset X$, then $\mathbf{1}_{E}$
denotes the characteristic function of $E$; that is, $\mathbf{1}_{E}(x)=1$ if $x\in E$
and $\mathbf{1}_{E}(x)=0$ if $x\in X\setminus E$.

\begin{definition}\label{d.D_structure}
Let $X$ be a metric measure space (recall \S\ref{s.metric}). 
Fix $1\le p<\infty$ and $0<\beta\le 1$.
Suppose that
for each $u\in\mathrm{Lip}_\beta(X)$, we are
given a family
$\mathcal{D}(u)\not=\emptyset$ of measurable functions  $X\to [0,\infty]$
as follows.
 First, we assume 
the following Poincar\'e inequality condition: 
\begin{itemize}
\item[(D1)]
There are constants $K>0$ and $\tau\ge 1$ such that the $(1,p)$-Poincar\'e inequality
\begin{equation}\label{e.poincare}
\vint_{B} \lvert u(x)-u_{B}\rvert\,d\mu(x)
\le K^{1/p}\mathrm{diam}(B)^{\beta}\bigg(\vint_{\tau B} g(x)^p \,d\mu(x)\bigg)^{1/p}
\end{equation}
holds whenever $B$ is a ball in $X$ and whenever 
$u\in\mathrm{Lip}_\beta(X)$ and $g\in \mathcal{D}(u)$.
\end{itemize}
Second, for all $\beta$-H\"older functions $u,v\colon X\to \R$, we assume the following
conditions (D2)--(D4):
\begin{itemize}
\item[(D2)] $\lvert a\rvert g\in \mathcal{D}(au)$ if $a\in\R$ and $g\in \mathcal{D}(u)$;
\item[(D3)] $g_u + g_v\in \mathcal{D}(u+v)$ if $g_u\in\mathcal{D}(u)$ and $g_v\in\mathcal{D}(v)$;
\item[(D4)] If $v\colon X\to \R$ is $\beta$-H\"older with a constant $\kappa\ge 0$ and 
$v|_{X\setminus E}=u|_{X\setminus E}$
for a Borel set $E\subset X$, then
$\kappa \mathbf{1}_{E} + g_u \mathbf{1}_{X\setminus E}\in\mathcal{D}(v)$
whenever $g_u\in \mathcal{D}(u)$.
\end{itemize}

Then we say that the family $\{\mathcal{D}(u)\,:\,u\in\mathrm{Lip}_\beta(X)\}$
is a {\em $\mathcal{D}$-structure in $X$,} with exponents 
$1\le p<\infty$ and $0<\beta \le 1$, and with constants $K>0$
and $\tau \ge 1$.
\end{definition}


Later in \S\ref{s.main} we will need
a stronger form of the condition (D1).
This stronger form (D1'), corresponding to a
$(p,p)$-Poincar\'e inequality, is explicitly stated in the following theorem.

\begin{theorem}\label{t.pp_poincare}
Suppose that 
$\{\mathcal{D}(u)\,:\,u\in\mathrm{Lip}_\beta(X)\}$
is a $\mathcal{D}$-structure in a geodesic space $X$, with exponents 
$1\le p<\infty$ and $0<\beta \le 1$, and with constants $K>0$
and $\tau \ge 1$. Then 
the following condition is valid:
\begin{itemize}
\item[\rm{(D1')}]
There exists $K_{p,p}=C(c_\mu,\beta,p,q,\tau)K>0$ such that the $(p,p)$-Poincar\'e inequality
\[
\bigg(\vint_{B} \lvert u(x)-u_{B}\rvert^p\,d\mu(x)\bigg)^{1/p}
\le K_{p,p}^{1/p}\mathrm{diam}(B)^{\beta}\bigg(\vint_{B} g(x)^p \,d\mu(x)\bigg)^{1/p}
\]
holds whenever $B$ is a ball in $X$ and whenever 
$u\in\mathrm{Lip}_\beta(X)$ and $g\in \mathcal{D}(u)$.
\end{itemize}
\end{theorem}

Theorem \ref{t.pp_poincare} is an immediate consequence of
a stronger result, namely the $\mathcal{D}$-structure independent Theorem \ref{t.qp_poincare}.
Moreover, the latter result gives $(q,p)$-Poincar\'e inequalities
for some $q>p$.
By formulating Theorem \ref{t.qp_poincare} separately, we wish to emphasize 
the contrast that $\mathcal{D}$-structures are not needed in this
`simpler' aspect of self-improvement.

\smallskip
We need a chaining lemma from \cite[p.~30--31]{MR1800917}.

\begin{lemma}\label{l.chain}
Suppose that $X$ is a geodesic space and that $\tau\ge 1$. 
Then
there are constants $M=C(\tau)\ge 1$ and $a=C(\tau)>1$
as follows.

Every ball $B\subset X$ contains a ball $B_0\subset B$ such that, for each $x\in B$, 
there is a sequence of balls $\{B_i\,:\,i=1,2,\ldots\}$ in $X$
satisfying the following conditions:
\begin{enumerate}[label=\rm{(\alph*)}]
\item\label{c1} $\tau B_i\subset B$ for all $i\ge 0$; 
\item\label{c2} $B_i$ is centered at $x$ for all sufficiently large $i$;
\item\label{c3} the radius $r_i$ of $B_i$ satisfies
$M^{-1} a^{-i} \diam(B) \le r_i \le M a^{-i} \diam(B)$
for all $i\ge 0$; and
\item\label{c4} the intersection $B_i\cap B_{i+1}$ contains
a ball $R_i$ such that $B_i\cup B_{i+1}\subset MR_i$ for all $i\ge 0$.
\end{enumerate}
\end{lemma}

%

We also need the following lemma, which is essentially
\cite[Lemma 4.22]{MR1800917}. See also \cite[p.~485]{MR807149}.

\begin{lemma}\label{l.kolmogorov}
Let $B\subset X$ be a ball in a metric space and let $u\colon B\to \R$
be a measurable function. Fix $1\le q<t<\infty$
and $C_0>0$ such that
\[
\mu(\{x\in B\,:\, \lvert u(x)\rvert > \lambda\})\le C_0 \lambda^{-t}
\]
for each $\lambda>0$. Then
\[
\bigg(\vint_B \lvert u\rvert^q\,d\mu\bigg)^{1/q}\le 2^{1/q}\bigg(\frac{C_0 q}{t-q}\bigg)^{1/t} \mu(B)^{-1/t}\,.
\]
\end{lemma}

The following self-improvement result follows from a straightforward adaptation of the main result in \cite{MR1336257}
that corresponds to the case $\beta=1$.
We refer to \cite{MR3089750}
for versions of this result taking place in general metric spaces and with any $\beta>0$. For convenience, we recall the proof.

\begin{theorem}\label{t.qp_poincare}
Suppose that $X$ is a geodesic space.
Fix  exponents
$1\le p<\infty$ and $0<\beta \le 1$.
Suppose that 
$u\in\mathrm{Lip}_\beta(X)$ and that $g\colon X\to [0,\infty]$ is
a measurable function.
Assume further that there are constants $K>0$ and $\tau\ge 1$ such that inequality
\[
\vint_{B} \lvert u(x)-u_{B}\rvert\,d\mu(x)
\le K^{1/p}\mathrm{diam}(B)^{\beta}\bigg(\vint_{\tau B} g(x)^p \,d\mu(x)\bigg)^{1/p}
\]
holds whenever $B$ is a ball in $X$.
Suppose that $Q\ge \log_2 c_\mu>0$
satisfies inequality $\beta p<Q$, where $c_\mu$ is the doubling constant
of $\mu$. Fix  $1\le q<Qp/(Q-\beta p)$.
Then there is a constant $C=C(c_\mu,Q,\beta,p,q,\tau)>0$
such that  inequality
\[
\bigg(\vint_B \lvert u(x)-u_B\rvert^q\,d\mu(x)\bigg)^{1/q}\le CK^{1/p} \diam(B)^\beta\bigg(\vint_B g(x)^p\,d\mu(x)\bigg)^{1/p}
\]
holds whenever $B\subset X$ is a ball.
\end{theorem}

\begin{proof}
Fix $u$, $g$ and a ball $B=B(x_0,r)\subset X$ with $r>0$.
Without loss of generality, we may assume that $r\le 2\diam(B)$.  
Let $B_0\subset B$ be the fixed ball as in Lemma \ref{l.chain} for the given $B\subset X$ and $\tau \ge 1$.
By subtracting
a constant from $u$, if necessary, we can assume
that $u_{B_0}=0$.

Let $\lambda >0$ and let
$x\in U^\lambda =\{y\in B\,:\, \lvert u(y)\rvert>\lambda\}$. 
Fix $\{B_i=B(x_i,r_i)\,:\,i=1,2,\ldots\}$ and $\{R_i\,:\, i=0,1,\ldots\}$
that are associated with the point $x$
and the ball $B$ as in Lemma \ref{l.chain}. In particular, the
properties \ref{c1}--\ref{c4} of the chain are valid. 
By the properties \ref{c2} and \ref{c3}, we have $u_{B_i}\to u(x)$ as $i\to \infty$, 
 and so 
\begin{align*}
\lambda&< \lvert u(x)\rvert = \lvert u(x)-u_{B_0}\rvert
\le \sum_{i=0}^\infty \lvert u_{B_{i+1}}-u_{B_i}\rvert\\
&\le \sum_{i=0}^\infty \big(\lvert u_{B_{i+1}}-u_{R_i}\rvert + \lvert u_{R_i}-u_{B_i}\rvert\big)\\
&\le \sum_{i=0}^\infty \bigg(\frac{\mu(B_{i+1})}{\mu(R_i)} \vint_{B_{i+1}} \lvert u-u_{B_{i+1}}\rvert\,d\mu+\frac{\mu(B_{i})}{\mu(R_i)} \vint_{B_{i}} \lvert u-u_{B_{i}}\rvert\,d\mu\bigg)\\
&\le CK^{1/p}\sum_{i=0}^\infty  r_i^\beta \bigg(\vint_{\tau B_i} g^p\,d\mu\bigg)^{1/p}\,.
\end{align*}
Hence for any $0<\varepsilon<1$, that is to be chosen later, we obtain that
\begin{align*}
 \sum_{i=0}^\infty  \lambda r^{-\beta\varepsilon} r_i^{\beta \varepsilon}
 \le C\lambda r^{-\beta\varepsilon}\sum_{i=0}^\infty  (a^{-i}\diam(B))^{\beta \varepsilon}\le C\lambda \le CK^{1/p}\sum_{i=0}^\infty r_i^\beta \bigg(\vint_{\tau B_i} g^p\,d\mu\bigg)^{1/p}\,.
\end{align*}
By comparing the sums on the left and right, we obtain an index $i_x\in \{0,1,\ldots\}$ such that
\[
\lambda r^{-\beta\varepsilon}r_{i_x}^{\beta \varepsilon}\le CK^{1/p} r_{i_x}^\beta \bigg(\vint_{\tau B_{i_x}} g^p\,d\mu\bigg)^{1/p}\,.
\]
A straightforward chaining argument, relying on the  properties \ref{c2}--\ref{c4} of the chain, implies that
$x\in C(M,a)B_{i_x}=B'_{i_x}$ for a constant $C(M,a)\ge 1$. 
By the previous estimates and property \ref{c1} of the chain,
\begin{equation}\label{e.penu}
\lambda ^p r_{i_x}^{\beta p (\varepsilon-1)}\mu(B_{i_x})\le CKr^{\beta  p\varepsilon}\int_{\tau B_{i_x}} g^p\,d\mu\le CKr^{\beta p\varepsilon}\int_{\tau B_{i_x}'} \mathbf{1}_{B}g^p\,d\mu\,.
\end{equation}
The assumptions on $Q$, inequality \eqref{e.radius_measure},
and properties \ref{c1} and \ref{c3} together imply that
\[
\frac{\mu(B_{i_x})}{\mu(B)}\ge \frac{\mu(M^{-1}B_{i_x})}{\mu(B)}\ge C\Big(\frac{r_{i_x}}{\diam(B)}\Big)^Q\ge C\Big(\frac{r_{i_x}}{r}\Big)^Q\,.
\]
By first raising this to power $\beta p(\varepsilon-1)/Q<0$ and  then substituting 
the result
to 
\eqref{e.penu},
\begin{equation}\label{e.ds}
\begin{split}
\lambda ^p \mu(5\tau B_{i_x}')^{1+\beta p (\varepsilon-1)/Q}&\le C\lambda ^p \mu(B_{i_x})^{1+\beta p (\varepsilon-1)/Q}\\&\le CK r^{\beta p}\mu(B)^{\beta p(\varepsilon-1)/Q}\int_{\tau B_{i_x}'} \mathbf{1}_{B}g^p\,d\mu\,.
\end{split}
\end{equation}

Using the $5r$-covering lemma \cite[Lemma 1.7]{MR2867756}, we obtain  a countable and disjoint subfamily
\[\{\tau B_{x_k}'\}\subset \{\tau B_{i_x}'\,:\, x\in U^\lambda\}\] of balls indexed by $k$ such that the covering property $U^\lambda\subset \cup_k 5\tau B_{x_k}'$ holds true.
Let us also observe that $0<1+ \beta p(\varepsilon-1)/Q<1$.
Hence, by the above covering property and \eqref{e.ds},
\begin{equation}\label{e.final}
\begin{split}
\lambda^p\mu(U^\lambda)^{1+\beta p (\varepsilon-1)/Q} 
&\le \sum_k \lambda^p\mu(5\tau B_{x_k}')^{1+\beta p(\varepsilon-1)/Q} 
\\
&\le CK r^{\beta p}\mu(B)^{\beta p(\varepsilon-1)/Q} \sum_k \int_{\tau B_{x_k}'} \mathbf{1}_{B}g^p\,d\mu
\\
&\le CK r^{\beta p}\mu(B)^{\beta p(\varepsilon-1)/Q} \int_{B} g^p\,d\mu\,.
\end{split}
\end{equation}

Recall that $\beta p< Q$ and $1\le q < Qp/(Q-\beta p)$. These facts allows us to choose the number
$0<\varepsilon<1$, depending on $Q$, $p$ and $\beta$ only, such that
$\max\{q,p\}< t =p/(1+\beta p(\varepsilon-1)/Q)$.
Thus, by 
raising inequality \eqref{e.final} to the power $t/p$
and applying Lemma \ref{l.kolmogorov},
we obtain
\begin{align*}
\bigg(\vint_{B} \lvert u-u_{B}\rvert^qd\mu\bigg)^{t/q}
&\le 2^t\bigg(\vint_{B} \lvert u\rvert^qd\mu\bigg)^{t/q}\\
& \le \frac{C\sup_{\lambda >0} \lambda^t\mu(\{x\in B\,:\,  \lvert u(x)\rvert>\lambda\})}{\mu(B)} \\
& \le C K^{t/p}r^{\beta t}\bigg(\vint_{B} g^p\,d\mu\bigg)^{t/p}\,.
\end{align*}
Since $B\subset X$ is an arbitrary ball,
we conclude the proof by raising both sides to power $1/t$
and recalling that $r\le 2\diam(B)$.
\end{proof}

\section{Boundedness results for maximal operators}\label{s.main}

\subsection{The main result}

Here we  formulate and prove our main result,
Theorem \ref{t.main_local}. This theorem can be viewed as 
a boundedness
result for a certain maximal function which, in turn, is  naturally associated
with a given $\mathcal{D}$-structure. More specifically, let $1<p<\infty$ and $0<\beta\le 1$.
If $\mathcal{B}\not=\emptyset$ is a given family
of balls in $X$, then we define a fractional sharp maximal function 
\begin{equation}\label{d.m_def}
M^{\sharp,p}_{\beta,\mathcal{B}}u(x)=\sup_{x\in B\in \mathcal{B}} \bigg(\frac{1}{\diam(B)^{\beta p}}\vint_B \lvert u(y)-u_B\rvert^p\,d\mu(y)\bigg)^{1/p}\,,\qquad x\in X\,,
\end{equation}
whenever $u\colon X\to \R$ is a $\beta$-H\"older function. 
The supremum above is defined to be zero, if there is no ball $B$ in $\mathcal{B}$ that contains the point $x$.

We are {\em primarily} interested in the 
localized maximal function $M^{\sharp,p}_{\beta,\mathcal{B}_0}u$ that is associated with the ball family
\begin{equation}\label{e.B_0}
\mathcal{B}_0=\{B\subset X\,:\, B\text{ is a ball such that }2B\subset B_0\}\,;
\end{equation}
here  and in the statement of Theorem \ref{t.main_local}, the set $B_0\subset X$ of localization is a fixed ball, and the case $X=B_0$ is allowed but then $X$ is of course necessarily bounded.

\begin{theorem}\label{t.main_local}
Suppose we are given  a $\mathcal{D}$-structure in a geodesic space $X$, with
exponents $1<p<\infty$ and $0<\beta\le 1$. 
Let $K_{p,p}>0$ be the constant for the $(p,p)$-Poincar\'e inequality 
 as in condition (D1') of Theorem \ref{t.pp_poincare}.
Let $k\in \N$, $0\le \varepsilon< p-1$,
 and $\alpha=\beta p^2/(2(s+\beta p))>0$ with $s=\log_2 c_\mu$. 
 Suppose that $B_0\subset X$ is a fixed ball.
Then inequality
\begin{equation}\label{e.loc_des}
\begin{split}
\int_{B_0}\big( M^{\sharp,p}_{\beta,\mathcal{B}_0} u\big)^{p-\varepsilon}\,d\mu
&\le C_1\bigg(2^{k(\varepsilon-\alpha)}+\frac{K_{p,p}4^{k\varepsilon}}{k^{p-1}}\bigg)\int_{B_0} \big( M^{\sharp,p}_{\beta,\mathcal{B}_0} u\big)^{p-\varepsilon}\,d\mu
\\&\qquad+C_1 C(k,\varepsilon)K_{p,p}\int_{B_0\setminus \{M^{\sharp,p}_{\beta,\mathcal{B}_0} u=0\}} g^p\big( M^{\sharp,p}_{\beta,\mathcal{B}_0} u\big)^{-\varepsilon}\,d\mu
\end{split}
\end{equation}
holds for each $u\in\mathrm{Lip}_\beta(X)$ and every $g\in\mathcal{D}(u)$. Here the constant
$C_1>0$ depends only on the parameters $\beta$, $p$, $c_\mu$; and $C(k,\varepsilon)=(4^{k\varepsilon}-1)/\varepsilon$ if 
$\varepsilon>0$ and $C(k,0)=k$.
\end{theorem}

Let us observe that the first term on the right-hand side of \eqref{e.loc_des}
is finite, since $u$ is assumed to be a $\beta$-H\"older function.
The following corollary is obtained when this term is absorbed to the left-hand side after choosing the numbers $k$ and $0\le \varepsilon<\varepsilon_0$
appropriately; for instance, we can choose $\varepsilon_0=1/k$ for a  large enough $k$. 

\begin{corollary}\label{c.main}
Suppose that we are given  a $\mathcal{D}$-structure in a geodesic space $X$, with
exponents $1<p<\infty$ and $0<\beta\le 1$.
Then
there exists some $0<\varepsilon_0<p-1$
with the property that  
for every $0\le \varepsilon<\varepsilon_0$ there is a constant $C>0$ such that inequality
\begin{equation}\label{e.loc_des_C}
\begin{split}
\int_{B_0}\big( M^{\sharp,p}_{\beta,\mathcal{B}_0} u\big)^{p-\varepsilon}\,d\mu \le 
C\int_{B_0\setminus \{M^{\sharp,p}_{\beta,\mathcal{B}_0} u=0\}} g^p\big( M^{\sharp,p}_{\beta,\mathcal{B}_0} u\big)^{-\varepsilon}\,d\mu
\end{split}
\end{equation}
holds whenever $B_0$ is a ball in $X$ and whenever $u\in\mathrm{Lip}_\beta(X)$ and $g\in\mathcal{D}(u)$.
\end{corollary}

\begin{question}\label{q.maximal}
Corollary \ref{c.main} suggests the following problem related to weighted
inequalities. Fix $1<p<\infty$ and $0<\beta \le 1$. Let us denote
by $\mathcal{B}$ the family of all balls in $X$.
Then, for some interesting  $\mathcal{D}$-structure, is it possible to
characterize those weights $w$ in $X$
for which inequality
\begin{equation*}
\begin{split}
\int_{X}\big( M^{\sharp,p}_{\beta,\mathcal{B}} u\big)^{p}\,w\,d\mu \le 
C\int_{X} g^p\,w\,d\mu
\end{split}
\end{equation*}
holds for each $u\in\mathrm{Lip}_\beta(X)$ and for every $g\in\mathcal{D}(u)$?
To our knowledge, this is an open problem even when $X=\R^n$ equipped with the Lebesgue measure.
\end{question}

\begin{remark}
It is instructive to reflect Question \ref{q.maximal} and Corollary \ref{c.main} by considering the following simple analogy
with $X=\R^n$ equipped with the Lebesgue measure.
If $1<p<\infty$, then the Muckenhoupt $A_p$ class consists precisely of weights $w$ for which the 
maximal operator
\[u\mapsto Mu=\sup\{\lvert u\rvert_B\mathbf{1}_B\,:\,B\subset \R^n\text{ is any ball}\,\}\]
is bounded on $L^p(w\,dx)$.
Whereas Question \ref{q.maximal} asks for a counterpart of this classical  result in the present setting,
Corollary \ref{c.main}, in turn, corresponds to a rather
curious special case. Namely, let $0\le \delta <1$ and let $u$ be a 
measurable function with $0<\lVert u\rVert_\infty<\infty$. Then $(Mu)^\delta$ is a Muckenhoupt $A_1$ weight
whose $A_1$-constant is independent of $u$; cf.\ \cite[Theorem 3.4 in \S2]{MR807149}.
As a consequence, the function $w=(Mu)^{-\varepsilon}$ is an $A_p$ weight
if $\varepsilon=\delta(p-1)>0$. Moreover,
the $A_p$ constant of this weight is independent of $u$.
By the boundedness
of the maximal function in $L^p(w\,dx)$,
and the fact that $w(x)\le \lvert u(x)\rvert^{-\varepsilon}$ almost everywhere, we find that
\begin{align*}
\int_{\R^n} \big(Mu(x)\big)^{p-\varepsilon}\,dx
&=\int_{\R^n} \big(Mu(x)\big)^p\,w(x)\,dx
\\&\le C\int_{\R^n} \lvert u(x)\rvert^p\, w(x)\,dx
\le C\int_{\R^n} \lvert u(x)\rvert^{p-\varepsilon}\,dx\,.
\end{align*}
In some cases, see \S\ref{s.Keith_Zhong} in particular, we can further adapt
this computation to the present setting.
\end{remark}

The proof of Theorem \ref{t.main_local} is completed in \S\ref{ss.main_local}. 
For the proof, we need preparations
that are treated in \S\ref{ss.Whitney} -- \S\ref{ss.auxiliary_local}.
At this stage, we already fix $X$, the $\mathcal{D}$-structure, $K_{p,p}$, $B_0\subsetneq X$, $\mathcal{B}_0$, $p$, $\beta$, $\varepsilon$, $k$ and $u$ as
in the statement of Theorem \ref{t.main_local}.
We refer to these objects throughout 
\S\ref{s.main}
without further notice.   Notice, however, that the function $g$ is not yet fixed.  

Let us emphasize that the ball $B_0$ in the proof below is further assumed to be a strict subset of $X$. That is, we will only focus on the case $B_0\not=X$. We remark that if $B_0=X$, then $X$ is bounded and the following Whitney cover 
$\mathcal{W}_0$ is replaced with the singleton $\{Q=B_0\}$. The other modifications in this easier special case are straightforward and we omit the details.

\subsection{Whitney ball covering}\label{ss.Whitney} We need a Whitney ball covering 
$\mathcal{W}_0=\mathcal{W}(B_0)$ of the ball $B_0\subsetneq X$.
This countable family with good covering properties
is comprised of the so-called {\em Whitney balls} that are of the form $Q=B(x_Q,r_Q)\in\mathcal{W}_0$,
with center $x_Q\in B_0$ and radius
\[
r_Q=\frac{\dist(x_Q, X\setminus B_0)}{128}>0\,.
\]
The $4$-dilated Whitney ball is denoted by $Q^*=4Q=B(x_Q,4 r_Q)$ whenever $Q\in\mathcal{W}_0$.
Even though the Whitney balls need not be pairwise disjoint, they nevertheless have 
the following standard covering properties with bounded overlap; cf.\ \cite[pp.~77--78]{MR2867756}.
\begin{itemize}
\item[(W1)] $B_0=\bigcup_{Q\in\mathcal{W}_0} Q$;
\item[(W2)] $\sum_{Q\in\mathcal{W}_0} \mathbf{1}_{Q^*}\le C\mathbf{1}_{B_0}$ for some constant $C=C(c_\mu)>0$.
\end{itemize}
The facts (W3)--(W6) below 
for any Whitney ball $Q=B(x_Q,r_Q)\in\mathcal{W}_0$
are straightforward to verify
by using inequality \eqref{e.diams} and the assumption $B_0\subsetneq X$; we omit the simple proofs. 
Below we refer to the family $\mathcal{B}_0$ of
balls that is defined in \eqref{e.B_0} by using the fixed ball $B_0$.
\begin{itemize} 
\item[(W3)] If $B\subset X$ is a ball such that $B\cap Q\not=\emptyset\not=2B\cap(X\setminus Q^*)$, then
$\diam(B)\ge 3r_Q/4$;
\item[(W4)] If  $B\subset Q^*$ is a ball, then $B\in\mathcal{B}_0$;
\item[(W5)] If $B\subset Q^*$ is a ball, $x\in B$  and $0<r\le \diam(B)$, then
$B(x,5r)\in\mathcal{B}_0$;
\item[(W6)] If $x\in Q^*$ and $0<r\le 2\diam(Q^*)$, then $B(x,r)\in\mathcal{B}_0$.
\end{itemize}
Observe that there is some overlap between
the properties (W4)--(W6). The slightly different formulations will conveniently 
guide the reader
in the sequel.

\subsection{Fractional sharp maximal functions}
We abbreviate
$M^{\sharp} u=M^{\sharp,p}_{\beta,\mathcal{B}_0} u$
and denote \[U^{\lambda}=\{x\in B_0\,:\,M^{\sharp} u(x)>\lambda\}\,,\qquad \lambda>0\,.\]
The sets $U^\lambda$ are open in $X$.
If $E\subset X$ is a Borel set and $\lambda>0$, we write $U^\lambda_E=U^\lambda \cap E$.
We also need a certain smaller maximal function that is localized to Whitney balls.
More specifically, for each  $Q\in\mathcal{W}_0$, we first consider the ball family\footnote{Let us emphasize that it is important to use $Q^*$ in the 
definition for $\mathcal{B}_Q$ instead of $Q$.
}
\[
\mathcal{B}_{Q}=\{B\subset X\,:\,  B  \text{ is a ball such that }B\subset Q^*\}
\]
and define  $M^\sharp_{Q} u=M^{\sharp,p}_{\beta,\mathcal{B}_{Q}} u$.
By using these individual maximal functions, we then define a {\em Whitney-ball localized sharp maximal function}\footnote{It is
equally important to use $\mathbf{1}_Q$ instead of $\mathbf{1}_{Q^*}$ in the
definition of $M^\sharp_{\mathrm{loc}} u$; these are delicate matters and related to the latter selection of stopping balls with the aid of condition (W3).}
\[
M^\sharp_{\textup{loc}} u = \sup_{Q\in\mathcal{W}_0} \mathbf{1}_Q M^\sharp_{Q} u\,.
\]
If $\lambda>0$ and $Q\in\mathcal{W}_0$, we write \begin{equation}\label{e.super}
Q^\lambda = \{x\in Q\,:\, M^\sharp_Q u(x)>\lambda\}
\quad\text{and}\quad V^{\lambda}=\{x\in B_0\,:\, M^\sharp_{\textup{loc}}u(x)>\lambda\}\,.
\end{equation}

We need the following
norm estimate between the different maximal functions. 
Its purpose, roughly speaking, is to create space for the forthcoming stopping balls in \S\ref{s.stopping} to expand, without losing their control in terms of $M^{\sharp} u$.
On the other hand, controlling this expansion is the only purpose for introducing the different maximal functions aside from $M^\sharp u$.

\begin{lemma}\label{l.big_to_small_ball}
There is a constant $C=C(c_\mu,p,\beta)\ge 1$ such that
\[
\int_{B_0} \big(M^{\sharp} u(x)\big)^{p-\varepsilon}\,d\mu(x)\le C\int_{B_0} \big(M^\sharp_{\textup{loc}} u(x)\big)^{p-\varepsilon}\,d\mu(x)\,.
\]
\end{lemma}

\begin{proof}
Recall that
\[
\int_{B_0} \big(M^\sharp(x)\big)^{p-\varepsilon}\,d\mu(x)
=(p-\varepsilon)\int_0^\infty \lambda^{p-\varepsilon}\mu(\{x\in B_0\,:\, M^\sharp u(x)>\lambda\})\frac{d\lambda}{\lambda}\,.
\]
By using also the corresponding identity for the  maximal function $M^\sharp_{\textup{loc}} u$, we see that it suffices to prove that inequality \begin{equation}\label{e.goal_big}
\mu(U^\lambda)\le C_1\mu(V^{\lambda/{C_1}})
\end{equation} holds for some $C_1=C(c_\mu,p,\beta)\ge 1$. Indeed, then
one can choose $C=C_1^{1+p}$.
We will now show how inequality \eqref{e.goal_big} follows
from an adaptation of \cite[Lemma 12.3.1]{MR3363168}. However, the
simple but tedious modification of the last rather short lemma is left to the interested reader. 

Fix $x\in B_0$ and let us consider
any ball $B=B(x_B,r_B)$ which satisfies the two conditions $x\in B$ and $256B=B(x_B,256r_B)\subset B_0$. By the covering condition (W1) there is a Whitney ball
$Q=B(x_Q,r_Q)\in\mathcal{W}_0$ such that  $x\in Q$. 
We claim that $B\subset Q^*$. 
In order to show this, we fix $y\in B\subset B(x,2r_B)$. Since
$B(x,255 r_B)\subset B_0$, we find that
\begin{align*}
d(y,x_Q)&\le d(y,x)+ d(x,x_Q)<2r_B+r_Q\\&\le \frac{2}{255}\cdot \dist(x,X\setminus B_0)+\frac{\dist(x_Q,X\setminus B_0)}{128}\\
&\le d(x,x_Q)+\frac{2}{255}\cdot \dist(x_Q,X\setminus B_0)+\frac{\dist(x_Q,X\setminus B_0)}{128}\\
&\le \frac{\dist(x_Q,X\setminus B_0)}{128}+\frac{2}{255}\cdot\dist(x_Q,X\setminus B_0)+\frac{\dist(x_Q,X\setminus B_0)}{128}\\
&< \frac{\dist(x_Q,X\setminus B_0)}{32}=4r_Q\,.
\end{align*}
It follows that $y\in 4Q=Q^*$. We have shown that $B\subset Q^*$, and therefore $x\in B\in\mathcal{B}_Q$.
Thus,
\[
M^\sharp_{\textup{loc}} u(x)\ge \mathbf{1}_Q(x)M^\sharp_Q u(x)\ge \bigg(\frac{1}{\diam(B)^{\beta p}}\vint_B \lvert u(y)-u_B\rvert^p\,d\mu(y)\bigg)^{1/p}\,.
\]
With the aid of this estimate, the distributional inequality \eqref{e.goal_big} follows from an adaptation of
\cite[Lemma 12.3.1]{MR3363168} that, in turn, is based upon \cite[Lemma 3.2.1]{MR2415381}.
\end{proof}

The following lemma is a slight variant of \cite[Lemma 3.6]{MR1681586}.

\begin{lemma}\label{l.arm_local}
Fix $\lambda>0$ and $Q\in\mathcal{W}_0$.  
Then inequality \[
\lvert u(x)-u(y)\rvert \le C(\beta,c_\mu) \lambda d(x,y)^\beta\]
holds whenever $x,y\in  Q^*\setminus U^{\lambda}$.
\end{lemma}

\begin{proof}
Let us remark that the property (W6) is used below without further notice.
Fix $\lambda>0$, $Q\in\mathcal{W}_0$ and $x,y\in  Q^*\setminus U^{\lambda}$.
Write $d=d(x,y)$.
Since $Q^*\subset B_0$, it suffices to prove that 
\begin{equation}\label{e.des_hld}
\lvert u(x)-u(y)\rvert \le C(\beta,c_\mu) d(x,y)^\beta
\big(M^{\sharp} u(x)+M^{\sharp} u(y)\big)\,.
\end{equation}
We first consider a point $z\in Q^*$ and a radius $0<r\le 2\diam(Q^*)$. Write
$B_i=B(z,2^{-i}r)\in\mathcal{B}_0$ for each $i\in \{0,1,\ldots\}$. Then, with the standard `telescoping' argument, see for instance the proof of \cite[Lemma 3.6]{MR1681586}, we obtain
\begin{align*}
\lvert u(z)-u_{B(z,r)}\rvert
&\le c_\mu \sum_{i=0}^\infty \vint_{B_i} \lvert u-u_{B_i}\rvert\,d\mu \\ 
&\le c_\mu \sum_{i=0}^\infty 2^{\beta(1-i)}r^\beta \bigg(\frac{1}{\diam(B_i)^{\beta p}}\vint_{B_i} \lvert u-u_{B_i}\rvert^p\,d\mu\bigg)^{1/p}\\
&\le c_\mu  M^{\sharp}u(z)\cdot \sum_{i=0}^\infty 2^{\beta(1-i)}r^\beta 
\le C(\beta,c_\mu)\, r^\beta M^{\sharp}u(z)\,.
\end{align*}
As a consequence, since $y\in Q^*$ and $0<d=d(x,y)\le \diam(Q^*)$, 
\begin{align*}
\lvert u(y)-u_{B(x,d)}\rvert &\le \lvert u(y)-u_{B(y,2d)}\rvert+\lvert u_{B(y,2d)}-u_{B(x,d)}\rvert\\
&\le C(\beta,c_\mu)\, d^\beta M^{\sharp}u(y) + \frac{\mu(B(y,2d))}{\mu(B(x,d))}\vint_{B(y,2d)} \lvert u-u_{B(y,2d)}\rvert\,d\mu\\
&\le C(\beta,c_\mu)\,d^\beta \Biggl[M^{\sharp}u(y) +  \Biggl(\frac{1}{\diam(B(y,2d))^{\beta p}}\vint_{B(y,2d)} \lvert u-u_{B(y,2d)}\rvert^p\,d\mu\bigg)^{1/p}\Biggr]\\
&\le C(\beta,c_\mu)\, d^\beta M^{\sharp}u(y)\,.
\end{align*}
It follows that
\begin{align*}
\lvert u(x)-u(y)\rvert&\le \lvert u(x)-u_{B(x,d)}\rvert + \lvert u_{B(x,d)}-u(y)\rvert
\le C(\beta,c_\mu)\, d^\beta \big( M^{\sharp}u(x)+ M^{\sharp}u(y) \big)\,,
\end{align*}
which is the desired inequality \eqref{e.des_hld}.
\end{proof}

\subsection{Stopping construction}\label{s.stopping}
The following stopping construction is needed for each Whitney ball separately.
Fix a Whitney ball $Q\in\mathcal{W}_0$.
The number
\[
\lambda_Q=\bigg(\frac{1}{\mathrm{diam}(Q^*)^{\beta p}} \vint_{Q^*} \lvert u(y)-u_{Q^*}\rvert^p\,d\mu(y)\bigg)^{1/p}
\]
serves as a certain treshold value.
Fix a level $\lambda>\lambda_Q/2$.
We will construct a stopping family $\mathcal{S}_\lambda(Q)$ of balls
whose $5$-dilations, in particular, cover the set $Q^\lambda$; recall the definition from \eqref{e.super}.
As a first step towards the stopping balls, let
$B\in\mathcal{B}_{Q}$ be such that $B\cap Q\not=\emptyset$. The {\em parent ball} of $B$ is then defined to be $\pi(B)=2B$ if $2B\subset Q^*$ and $\pi(B)=Q^*$ otherwise.
Observe that $B\subset \pi(B)\in\mathcal{B}_Q$ and $\pi(B)\cap Q\not=\emptyset$ so that
the grandparent $\pi(\pi(B))$ is well defined, and so on and so forth.
Moreover, by inequalities \eqref{e.doubling} and \eqref{e.diams}, and property (W3) if needed, we have $\mu(\pi(B))\le c_\mu^5 \mu(B)$
and $\diam(\pi(B))\le 16\diam(B)$.

Now we come to the actual stopping argument. Let us fix a point $x\in Q^\lambda\subset Q$.
If $\lambda_Q/2<\lambda<\lambda_Q$, then we choose
$B_x=Q^*\in\mathcal{B}_{Q}$. If 
$\lambda\ge \lambda_Q$, then by using the condition $x\in Q^\lambda$ we first choose a starting ball $B$, with $x\in B\in\mathcal{B}_Q$, such that
\[
\lambda <
\bigg(\frac{1}{\mathrm{diam}(B)^{\beta p}} \vint_{B} \lvert u(y)-u_{B}\rvert^p\,d\mu(y)\bigg)^{1/p}\,.
\]
We continue by looking at the balls $B\subset \pi(B) \subset \pi(\pi(B))\subset \dotsb$ 
and we 
stop at the first ball among them, denoted by $B_x\in\mathcal{B}_{Q}$, that satisfies the following stopping conditions: 
\begin{align*}
\begin{cases} 
\lambda <
\displaystyle\bigg(\frac{1}{\mathrm{diam}(B_x)^{\beta p}} \vint_{B_x} \lvert u(y)-u_{B_x}\rvert^p\,d\mu(y)\bigg)^{1/p}\\
\displaystyle \bigg(\frac{1}{\mathrm{diam}(\pi(B_x))^{\beta p}} \vint_{\pi(B_x)} \lvert u(y)-u_{\pi(B_x)}\rvert^p\,d\mu(y)\bigg)^{1/p}\le \lambda \,.
 \end{cases}
\end{align*}
The inequality $\lambda\ge \lambda_Q$ in combination with assumption $B_0\subsetneq X$ ensures that there always is such a stopping ball.
In both cases above,  the chosen ball $B_x^\lambda=B_x\in\mathcal{B}_Q$  contains
the point $x$ and satisfies inequalities
\begin{equation}\label{e.loc_stop}
\lambda<
\bigg(\frac{1}{\mathrm{diam}(B_x^\lambda)^{\beta p}} \vint_{B_x^\lambda} \lvert u(y)-u_{B_x^\lambda}\rvert^p\,d\mu(y)\bigg)^{1/p}\le 32c_\mu^{5/p} \lambda\,.
\end{equation}
Now, by using the $5r$-covering lemma,  we obtain a countable disjoint family 
\[\mathcal{S}_\lambda(Q)\subset \{B_x^\lambda\,:\, x\in Q^\lambda\}\,,\qquad \lambda>\lambda_Q/2\,,\]
of {\em stopping balls}  such that $Q^\lambda\subset \cup_{B\in\mathcal{S}_\lambda(Q)}
5B$.
Let us remark that, by the condition (W4) and stopping inequality \eqref{e.loc_stop}, we have $B\subset U^{\lambda}_{Q^*}=U^{\lambda}\cap Q^*$ 
if
$B\in \mathcal{S}_\lambda(Q)$ and $\lambda>\lambda_Q/2$.

\subsection{Auxiliary local results}\label{ss.auxiliary_local}

We prove
two technical results: Lemma \ref{l.dyadic} and Lemma \ref{l.mainl_local}.
Even though 
the following lemma is a counterpart of \cite[Lemma 3.1.2]{MR2415381},
the adaptation to our setting is non-trivial.
Recall that $k\in\N$ is a fixed number and
$\alpha=\beta p^2/(2(s+\beta p))>0$ with $s=\log_2 c_\mu> 0$.

\begin{lemma}\label{l.dyadic}
Suppose that $Q\in\mathcal{W}_0$ and let $\lambda>\lambda_Q/2$. Then inequality
\begin{equation}\label{e.abso}
\begin{split}
&\frac{1}{\mathrm{diam}(B)^{\beta p}}\int_{U_{B}^{2^k\lambda}}\lvert u(x)-u_{B\setminus U^{2^k\lambda}}\rvert^p\,d\mu(x)\\&\le C(p,c_\mu)2^{-k\alpha} (2^{k}\lambda)^p 
\mu(U_{B}^{2^k\lambda})
+\frac{C(p,c_\mu)}{\mathrm{diam}(B)^{\beta p}}\int_{B\setminus U^{2^k\lambda}} \lvert u(x)-u_{B\setminus U^{2^k\lambda}}\rvert^p\,d\mu(x)
\end{split}
\end{equation}
holds whenever $B\in\mathcal{S}_\lambda(Q)$
is such that $\mu (U_{B}^{2^k\lambda}) < \mu(B)/2$.
\end{lemma}

\begin{proof}
Fix $\lambda>\lambda_Q/2$ and  let $B\in\mathcal{S}_\lambda(Q)$ be such that
$\mu (U_{B}^{2^k\lambda}) < \mu(B)/2$. Fix $x\in U_{B}^{2^k\lambda}\subset B$. Consider
the function $h\colon (0,\infty)\to\R$, 
\[
r\mapsto h(r)= \frac{\mu(U_{B}^{2^k\lambda}\cap B(x,r))}{\mu(B\cap B(x,r))}=\frac{\mu(U_{B}^{2^k\lambda}\cap B(x,r))}{\mu(B(x,r))}\cdot \bigg(\frac{\mu(B\cap B(x,r))}{\mu(B(x,r))}\bigg)^{-1}\,.
\]
By Lemma \ref{l.continuous} and the fact that $B$ is open, we find that $h\colon (0,\infty)\to \R$ is continuous.
Since $h(r)=1$ for small values of $r>0$,
and $h(r)<1/2$ for  $r>\diam(B)$, we find
that $h(r_x)=1/2$ for some $0<r_x\le \diam(B)$.
We write $B'_x=B(x,r_x)$. Then
\begin{equation}\label{e.tok}
\frac{\mu(U_{B}^{2^k\lambda}\cap  B'_x)}{\mu(B\cap B'_x)}=h(r_x)=\frac{1}{2}
\end{equation}
and
\begin{equation}\label{e.ens}
\frac{\mu((B\setminus U^{2^k\lambda})\cap  B'_x)}{\mu(B\cap  B'_x)}
=1-\frac{\mu(U_{B}^{2^k\lambda}\cap  B'_x)}{\mu(B\cap B'_x)}
= 1-h(r_x)=\frac{1}{2}\,.
\end{equation}
Let $\mathcal{G}_\lambda$ be a countable disjoint subfamily of $\{ B'_x\,:\,  x\in U_{B}^{2^k\lambda}\}$
such that $U_{B}^{2^k\lambda}\subset \cup_{B'\in\mathcal{G}_\lambda} 
5B'$.
Then 
\eqref{e.tok} and \eqref{e.ens}  hold for every ball $B'\in
\mathcal{G}_\lambda$; indeed, 
by denoting $B'_I=U_{B}^{2^k\lambda}\cap B'$ and
${B'_O}=(B\setminus U^{2^k\lambda})\cap B'$,
we have the following comparison identities:
\begin{equation}\label{e.comparison}
\mu(B'_I)=  \frac{\mu( B\cap B')}{2}=  
\mu({B'_O})\,,
\end{equation}
where all the measures are strictly positive. These identities
are important and they  are used several times throughout the remainder of this proof.

We multiply the left-hand side of \eqref{e.abso}
by $\diam(B)^{\beta p}$ and then estimate as follows:
\begin{equation}\label{e.prepare}
\begin{split}
\int_{U_{B}^{2^k\lambda}}  &\lvert u-u_{B\setminus U^{2^k\lambda}}\rvert^p\,d\mu
\le\sum_{B'\in\mathcal{G}_\lambda} \int_{5B'\cap B}\lvert u-u_{B\setminus U^{2^k\lambda}}\rvert^p\,d\mu\\
&\le 2^{p-1}\sum_{B'\in\mathcal{G}_\lambda} \mu(5B'\cap B) \lvert u_{{B'_O}}-u_{B\setminus U^{2^k\lambda}}\rvert^p+
2^{p-1}\sum_{B'\in\mathcal{G}_\lambda} \int_{5B'\cap B}\lvert u-u_{{B'_O}}\rvert^p\,d\mu\,.
\end{split}
\end{equation}
By  \eqref{e.doubling} and Lemma \ref{l.ball_measures},
we find that
 $\mu(5B'\cap B)\le \mu(8B') \le c_\mu^6 \mu(B\cap B')$
 if $B'\in\mathcal{G}_\lambda$.
Hence, by the comparison identities \eqref{e.comparison},
 \begin{equation}\label{e.eka}
\begin{split}
2^{p-1}&\sum_{B'\in\mathcal{G}_\lambda} \mu(5B'\cap B)  \lvert u_{{B'_O}}-u_{B\setminus U^{2^k\lambda}}\rvert^p
\le C(p,c_\mu) \sum_{B'\in\mathcal{G}_\lambda} \mu({B'_O})  
\vint_{{B'_O}} \lvert u-u_{B\setminus U^{2^k\lambda}}\rvert^p\,d\mu
\\&=C(p,c_\mu)\sum_{B'\in\mathcal{G}_\lambda}  
\int_{{B'_O}} \lvert u-u_{B\setminus U^{2^k\lambda}}\rvert^p\,d\mu
\le C(p,c_\mu)
\int_{B\setminus U^{2^k\lambda}} \lvert u-u_{B\setminus U^{2^k\lambda}}\rvert^p\,d\mu\,.
\end{split}
\end{equation}
This concludes our analysis of the `easy term' in
\eqref{e.prepare}.
In order to treat the remaining term therein, we do need some preparations.

Let us fix a ball $B'\in\mathcal{G}_\lambda$ that satisfies
$\int_{5B'\cap B} \lvert u-u_{{B'_O}}\rvert^p\,d\mu\not=0$. 
We claim that
\begin{equation}\label{e.out}
\vint_{5B'\cap B}\lvert u-u_{{B'_O}}\rvert^p\,d\mu\le  C(p,c_\mu) 2^{-k\alpha } (2^k\lambda)^p \diam(B)^{\beta p}\,.
\end{equation}
In order to prove this inequality, we fix a number $m\in \R$ such that 
\begin{align*}
(2^m \lambda)^p \diam(5B')^{\beta p}&=\vint_{5B'\cap B}\lvert u-u_{{B'_O}}\rvert^p\,d\mu\,.
\end{align*}
Let us first consider the case $m< k/2$. Then $m-k<-k/2$, and since always $\alpha<p/2$, the desired inequality \eqref{e.out} is obtained in this case as follows:
\begin{align*}
\vint_{5B'\cap B}\lvert u-u_{{B'_O}}\rvert^p\,d\mu
&=2^{(m-k)p} (2^k\lambda)^p\diam(5B')^{\beta p} \\&\le 10^p\, 2^{-kp/2}(2^k\lambda)^p\diam(B)^{\beta p}
\le C(p)  2^{-k\alpha }(2^k\lambda)^p\diam(B)^{\beta p}\,.
\end{align*}

Next we consider the case $k/2\le m$.
By comparison identities \eqref{e.comparison} and Lemma \ref{l.ball_measures},
\begin{align*}
\vint_{5B'\cap B} \lvert u-u_{{B'_O}}\rvert^p\,d\mu &\le  
2^{p-1}\vint_{5B'\cap B} \lvert u-u_{5B'}\rvert^p\,d\mu + 2^{p-1}\lvert 
u_{5B'}-u_{{B'_O}}\rvert^p
\\&\le 2^{p+1}c_\mu^6 \vint_{5B'} \lvert u-u_{5B'}\rvert^p\,d\mu
\le 2^{p+1}c_\mu^6 (2^k\lambda)^p\mathrm{diam}(5B')^{\beta p}\,,
\end{align*}
where the last step follows from condition (W5) and the fact that $5B'\supset {B'_O}\not=\emptyset$.
It follows that 
$2^{mp}\le 2^{p+1} c_\mu^6 2^{kp}$.
On the other hand, we have
\begin{align*}
(2^m \lambda)^p \diam(5B')^{\beta p}\mu(B'\cap B)&\le 
\int_{5B'\cap B} \lvert u-u_{{B'_O}}\rvert^p\,d\mu\\
&\le  2^{p-1}\int_{5B'\cap B} \lvert u-u_B\rvert^p\,d\mu
+ 2^{p-1}\mu(5B'\cap B) \lvert u_{{B'_O}}-u_B\rvert^p\\
&\le 2^{p+1} c_\mu^6 \int_B \lvert u-u_B\rvert^p\,d\mu\le 2 \cdot 64^p c_\mu^{11} \lambda^p \diam(B)^{\beta p}\mu(B)\,,
\end{align*}
where the last step follows from the fact that $B\in\mathcal{S}_\lambda(Q)$ in combination with
inequality \eqref{e.loc_stop}.
In particular, if $s=\log_2 c_\mu$ then by inequality \eqref{e.radius_measure}
and Lemma \ref{l.ball_measures}, we obtain that
\begin{align*}
\bigg(\frac{\diam(5B')}{\diam(B)}\bigg)^{s+\beta p} &\le 20^s \frac{\diam(5B')^{\beta p}\mu(B')}{\diam(B)^{\beta p}\mu(B)}
\le 20^s c_\mu^3\frac{\diam(5B')^{\beta p}\mu(B'\cap B)}{\diam(B)^{\beta p}\mu(B)}\\
&\le  2\cdot 64^p 20^s c_\mu^{14} 2^{-mp}\le 2\cdot 64^p 20^s c_\mu^{14} 2^{-kp/2}\,.
\end{align*}
This, in turn, implies that
\[
\bigg(\frac{\diam(5B')}{\diam(B)}\bigg)^{\beta p} \le 2\cdot 64^p 20^s c_\mu^{14} 2^{\frac{-k\beta p^2}{2(s+\beta p)}}=  C(p,c_\mu) 2^{-k\alpha}\,.
\]
Combining the above estimates, we see that
\[
\vint_{5B'\cap B}\lvert u-u_{{B'_O}}\rvert^p\,d\mu=(2^m \lambda)^p \diam(5B')^{\beta p}\le C(p,c_\mu) 2^{-k\alpha } (2^k\lambda)^p \diam(B)^{\beta p}\,.
\]
That is, inequality \eqref{e.out} holds also in the present case $k/2\le m$.

By using Lemma \ref{l.ball_measures} and inequalities \eqref{e.comparison} and  \eqref{e.out}, we can now estimate the second term in \eqref{e.prepare} as follows:
\begin{align*}
2^{p-1}\sum_{B'\in\mathcal{G}_\lambda} \int_{5B'\cap B}\lvert u-u_{{B'_O}}\rvert^p\,d\mu
&\le 2^p c_\mu^6\sum_{B'\in\mathcal{G}_\lambda} \mu(B'_I) \vint_{5B'\cap B}\lvert u-u_{{B'_O}}\rvert^p\,d\mu\\
&\le  C(p,c_\mu) 2^{-k\alpha } (2^k\lambda)^p \diam(B)^{\beta p}\sum_{B'\in\mathcal{G}_\lambda} \mu(B'_I)  
\\&\le C(p,c_\mu)2^{-k\alpha } (2^k\lambda)^p \diam(B)^{\beta p}\mu(U^{2^k\lambda}_B)\,. 
\end{align*}
Inequality \eqref{e.abso} follows by collecting the above estimates.
\end{proof}

The following lemma is essential for the proof of Theorem \ref{t.main_local}, and it is the only place
in the proof where
the $(p,p)$-Poincar\'e inequality is needed---and, moreover, this inequality
is applied
only a single time.

\begin{lemma}\label{l.mainl_local}
Fix a Whitney ball $Q\in\mathcal{W}_0$. Then
inequality 
\begin{equation}\label{e.id}
\begin{split}
\lambda^p \mu(Q^\lambda)
\le C(\beta,p,c_\mu)\biggl[\frac{(\lambda 2^{k})^p}{2^{k\alpha}} \mu(U^{2^k \lambda}_{Q^*})+   \frac{K_{p,p}}{k^p} \sum_{j=k}^{2k-1}  (\lambda 2^{j})^p \mu(U^{2^j \lambda}_{Q^*})
+ K_{p,p}\int_{U^{\lambda}_{Q^*}\setminus U^{4^k\lambda}} g^p\,d\mu\biggr]
\end{split}
\end{equation}
holds for each $\lambda>\lambda_Q/2$ and every $g\in\mathcal{D}(u)$.
\end{lemma}

\begin{proof}
Fix $\lambda>\lambda_Q/2$ and $g\in\mathcal{D}(u)$.
By the doubling condition \eqref{e.doubling},
\[\lambda^p \mu(Q^\lambda)
\le \lambda^p \sum_{B\in\mathcal{S}_\lambda(Q)}
\mu(5B) \le c_\mu^3 \sum_{B\in\mathcal{S}_\lambda(Q)} \lambda^p \mu(B)\,.\]
Recall also that $B\subset U^{\lambda}_{Q^*}=U^\lambda\cap Q^*$ if $B\in\mathcal{S}_\lambda(Q)$.
Therefore, and using the fact that $\mathcal{S}_\lambda(Q)$ is a disjoint family, it suffices
to prove that inequality
\begin{equation}\label{e.local}
\begin{split}
\lambda^p \mu(B)
\le C(\beta,p,c_\mu)\biggl[\frac{(\lambda 2^{k})^p}{2^{k\alpha}} \mu(U^{2^k \lambda}_{B})+
\frac{K_{p,p}}{k^p} \sum_{j=k}^{2k-1}  (\lambda 2^{j})^p \mu(U_B^{2^j \lambda})
+K_{p,p}\int_{B\setminus U^{4^k\lambda}} g^p\,d\mu\biggl]
\end{split}
\end{equation}
holds for every $B\in\mathcal{S}_\lambda(Q)$.
To this end, let us fix a ball $B\in\mathcal{S}_\lambda(Q)$.

If $\mu(U_B^{2^k\lambda})\ge \mu(B)/2$, then
\[
\lambda^p \mu(B) \le 2\lambda^p \mu(U_B^{2^k\lambda})
=2\frac{(\lambda 2^k)^p}{2^{kp}}\mu(U_B^{2^k\lambda})
\le 2\frac{(\lambda 2^k)^p}{2^{k \alpha}}\mu(U_B^{2^k\lambda})\,,
\]
which suffices for the required local estimate \eqref{e.local}.
Let us then consider the more difficult case $\mu(U_B^{2^k\lambda}) < \mu(B)/2$.
In this case, by the stopping inequality \eqref{e.loc_stop},
\begin{align*}
\lambda^p\mu(B)&\le \frac{1}{\mathrm{diam}(B)^{\beta p}}\int_{B} \lvert u(x)-u_{B}\rvert^p\,d\mu(x)
\\&\le \frac{2^{p}}{\mathrm{diam}(B)^{\beta p}}\int_{X} \Bigl( \mathbf{1}_{B\setminus U^{2^k\lambda}}(x)+
\mathbf{1}_{U^{2^k\lambda}_{B}}(x)\Bigr)\lvert u(x)-u_{B\setminus U^{2^k\lambda}}\rvert^p\,d\mu(x)\,.
\end{align*}
By Lemma
\ref{l.dyadic} it suffices to estimate 
the integral over the set
$B\setminus U^{2^k\lambda}=B\setminus U^{2^k\lambda}_B$;
observe that the measure of this set is strictly positive.
We remark that the Poincar\'e
inequality condition (D1') will be used to estimate this integral.

Fix a number $i\in\N$.
Recall that $B\subset Q^*$. Hence, it follows from Lemma \ref{l.arm_local} that the restriction $u|_{B\setminus U^{2^i\lambda}}\colon B\setminus U^{2^i\lambda}\to \R$ is a 
$\beta$-H\"older function with a constant $\kappa_i=C(\beta,c_\mu)2^i\lambda$. 
We can now use the McShane extension \eqref{McShane} and extend $u|_{B\setminus U^{2^i\lambda}}$ 
to a function $u_{2^i \lambda}\colon X\to \R$
that is $\beta$-H\"older with the constant $\kappa_i$
and satisfies the restriction identity 
$u_{2^i \lambda}|_{B\setminus U^{2^i \lambda}}=u|_{B\setminus U^{2^i\lambda}}$.

The crucial idea that was also used by Keith--Zhong in \cite{MR2415381} is to consider the function
\[h(x)=\frac{1}{k} \sum_{i=k}^{2k-1} u_{2^i \lambda}(x)\,,\qquad x\in X\,.\]
By
conditions (D2)--(D4) of the fixed $\mathcal{D}$-structure, we obtain that
\[
g_h=\frac{1}{k}\sum_{i=k}^{2k-1} \Bigl( \kappa_i  \mathbf{1}_{U^{2^i\lambda}\cup B^c}
+ g\mathbf{1}_{B\setminus U^{2^i\lambda}}\Bigr)\in\mathcal{D}(h)\,.
\]
Observe that $U^{2^k\lambda}_B\supset U^{2^{(k+1)}\lambda}_B
\supset \dotsb \supset U^{2^{(2k-1)}\lambda}_B\supset U^{4^{k}\lambda}_B$.
By using these inclusions it is straightforward to show that the following
pointwise estimates are valid in $X$,
\begin{equation*}\label{e.hardy}
\begin{split}
\mathbf{1}_Bg_h^p &\le \bigg( \frac{1}{k}\sum_{i=k}^{2k-1} \Bigl(   \kappa_i\,\mathbf{1}_{U_B^{2^i\lambda}}
+  g \mathbf{1}_{B\setminus U^{2^i\lambda}}\Bigr)\bigg)^p\\
&\le 2^{p}\bigg(\frac{1}{k}\sum_{i=k}^{2k-1} \kappa_i\, \mathbf{1}_{U_B^{2^i \lambda}}\bigg)^p
+  2^{p} g^p \mathbf{1}_{B\setminus U^{4^k\lambda}}\\
&\le \frac{C(\beta,p,c_\mu)}{k^p} \sum_{j=k}^{2k-1} \bigg(\sum_{i=k}^j  2^i \lambda\bigg)^p  \mathbf{1}_{U_B^{2^j \lambda}}
+  2^p g^p \mathbf{1}_{B\setminus U^{4^k\lambda}}\\
&\le \frac{C(\beta,p,c_\mu)}{k^p} \sum_{j=k}^{2k-1}  (\lambda 2^{j})^p  \mathbf{1}_{U_B^{2^j \lambda}}
+  2^p g^p \mathbf{1}_{B\setminus U^{4^k\lambda}}\,.
\end{split}
\end{equation*}
Observe that $h$ coincides with $u$ on $B\setminus U^{2^k\lambda}$
and recall that $g_h\in\mathcal{D}(h)$.
Hence
the Poincar\'e inequality from condition (D1')
in Theorem \ref{t.pp_poincare} implies 
that
\begin{align*}
&\frac{1}{\mathrm{diam}(B)^{\beta p}}\int_{B\setminus U^{2^k\lambda}}\lvert u(x)-u_{B\setminus U^{2^k\lambda}}\rvert^p\,d\mu(x)
\le \frac{2^{p}}{\mathrm{diam}(B)^{\beta p}}\int_{B} \lvert h(x)-h_{B}\rvert^p\,d\mu(x)
\\&\le 2^{p}K_{p,p}\int_{B} g_h(x)^p \,d\mu(x)
\\&\le \frac{C(\beta,p,c_\mu)K_{p,p}}{k^p} \sum_{j=k}^{2k-1}  (\lambda 2^{j})^p \mu(U_B^{2^j \lambda})
+4^p K_{p,p}\int_{B\setminus U^{4^k\lambda}} g(x)^p\,d\mu(x)\,.
\end{align*}
The desired local inequality \eqref{e.local} follows by combining the estimates above.
\end{proof}

\subsection{Completing proof of Theorem \ref{t.main_local}}\label{ss.main_local}
Recall that  $u\colon X\to \R$ is a $\beta$-H\"older function and
 that $M^{\sharp} u=M^{\sharp,p}_{\beta,\mathcal{B}_0}u$.
Let us fix a function $g\in\mathcal{D}(u)$.
Observe that the left-hand side of 
inequality \eqref{e.loc_des} is finite. Without loss of generality, we may further assume
that it is nonzero.
By Lemma \ref{l.big_to_small_ball},
\[
\int_{B_0} \big(M^{\sharp} u(x)\big)^{p-\varepsilon}\,d\mu(x)
\le C(c_\mu,p,\beta)\int_{B_0} \big(M^{\sharp}_{\textup{loc}} u(x)\big)^{p-\varepsilon}\,d\mu(x)\,.
\]
Observe that
\[
\big(M^{\sharp}_{\textup{loc}} u(x)\big)^{p-\varepsilon}\le \sum_{Q\in\mathcal{W}_0} \mathbf{1}_Q(x)\big(M^\sharp_Q u(x)\big)^{p-\varepsilon}
\]
for every $x\in B_0$. Hence,
\begin{align*}
\int_{B_0}  \big(M^{\sharp}_{\textup{loc}} u(x)\big)^{p-\varepsilon}\,d\mu(x)
\le \sum_{Q\in\mathcal{W}_0} \int_{Q} \big(M^{\sharp}_{Q} u(x)\big)^{p-\varepsilon}\,d\mu(x)\,.
\end{align*}
At this stage, we fix a ball $Q\in\mathcal{W}_0$ and write  
the corresponding integral as follows:
\begin{align*}
\int_{Q} \big(M^{\sharp}_{Q} u(x)\big)^{p-\varepsilon}\,d\mu(x) =
(p-\varepsilon)\int_0^\infty \lambda^{p-\varepsilon} \mu(Q^\lambda)\,\frac{d\lambda}{\lambda}\,.
\end{align*}
Since $Q^\lambda=Q=Q^{2\lambda}$ for every $\lambda \in (0,\lambda_Q/2)$, 
we find that
\begin{align*}
(p-\varepsilon)\int_0^{\lambda_Q/2} \lambda^{p-\varepsilon} \mu(Q^\lambda)\,\frac{d\lambda}{\lambda}&=\frac{(p-\varepsilon)}{2^{p-\varepsilon}}\int_0^{\lambda_Q/2} (2\lambda)^{p-\varepsilon} \mu(Q^{2\lambda})\,\frac{d\lambda}{\lambda}\\
&\le \frac{(p-\varepsilon)}{2^{p-\varepsilon}}\int_0^{\infty} \sigma^{p-\varepsilon} \mu(Q^{\sigma})\,\frac{d\sigma}{\sigma}
\\&=\frac{1}{2^{p-\varepsilon}}\int_Q \big(M^{\sharp}_{Q} u(x)\big)^{p-\varepsilon}\,d\mu(x)\,.
\end{align*}
On the other hand, by Lemma \ref{l.mainl_local}, for each $\lambda>\lambda_Q/2$,
\begin{align*}
\lambda^{p-\varepsilon} \mu(Q^\lambda)
\le C(\beta,p,c_\mu)\lambda^{-\varepsilon}\biggl[\frac{(\lambda 2^{k})^p}{2^{k\alpha}} \mu(U^{2^k \lambda}_{Q^*})+ \frac{K_{p,p}}{k^p} \sum_{j=k}^{2k-1}  (\lambda 2^{j})^p \mu(U^{2^j \lambda}_{Q^*})
+K_{p,p}\int_{U^{\lambda}_{Q^*}\setminus U^{4^k\lambda}} g^p\,d\mu\,\biggr].
\end{align*}
Since $p-\varepsilon>1$, it follows that
\begin{align*}
\int_Q \big( M^{\sharp}_{Q} u(x)\big)^{p-\varepsilon}\,d\mu(x)
&\le 2(p-\varepsilon)\int_{\lambda_Q/2}^\infty \lambda^{p-\varepsilon} \mu(Q^\lambda)\,\frac{d\lambda}{\lambda}
\\&\le C(\beta,p,c_\mu)(I_1(Q) + I_2(Q) + I_3(Q))\,,
\end{align*}
where
\begin{align*}
I_1(Q)&=\frac{2^{k\varepsilon}}{2^{k\alpha}}\int_{0}^\infty (\lambda 2^{k})^{p-\varepsilon} \mu(U^{2^k \lambda}_{Q^*})  \,\frac{d\lambda}{\lambda}\,,\qquad \\
I_2(Q)&= \frac{K_{p,p}}{k^p}\sum_{j=k}^{2k-1} 2^{j\varepsilon}\int_0^\infty (2^j \lambda)^{p-\varepsilon} \mu(U^{2^j \lambda}_{Q^*})\,\frac{d\lambda}{\lambda}\,, \\
I_3(Q)&= K_{p,p}
\int_0^\infty \lambda^{-\varepsilon} \int_{U^{\lambda}_{Q^*}\setminus U^{4^k\lambda}} g(x)^p\,d\mu(x)\,\frac{d\lambda}{\lambda}\,.
\end{align*}
By (W2) we have $\sum_{Q\in\mathcal{W}_0} \mathbf{1}_{Q^*}\le C(c_\mu)\mathbf{1}_{B_0}$. Hence, we can now continue to estimate as follows. First,
\begin{align*}
\sum_{Q\in\mathcal{W}_0} I_1(Q)
&\le C(c_\mu)\frac{2^{k(\varepsilon-\alpha)}}{p-\varepsilon}\int_{B_0} \big(M^{\sharp} u(x)\big)^{p-\varepsilon}\,d\mu(x)
\\
&\le C(c_\mu) 2^{k(\varepsilon-\alpha)}\int_{B_0} \big(M^{\sharp} u(x)\big)^{p-\varepsilon}\,d\mu(x)\,.
\end{align*}
Second, 
\begin{align*}
\sum_{Q\in\mathcal{W}_0} I_2(Q)
&\le C(c_\mu)\frac{K_{p,p}}{k^p}\sum_{j=k}^{2k-1}  2^{j\varepsilon}\int_0^\infty (2^j \lambda)^{p-\varepsilon} \mu(U^{2^j \lambda})\,\frac{d\lambda}{\lambda}\\
&\le \frac{C(c_\mu)K_{p,p}}{k^p(p-\varepsilon)}\bigg(\sum_{j=k}^{2k-1} 2^{j\varepsilon}\bigg)\int_{B_0} \big(M^{\sharp} u(x)\big)^{p-\varepsilon}\,d\mu
\\
&\le C(c_\mu) \frac{K_{p,p}4^{k\varepsilon}}{k^{p-1}}\int_{B_0} \big(M^{\sharp} u(x)\big)^{p-\varepsilon}\,d\mu\,.
\end{align*}
Third, by Fubini's theorem,
\begin{align*}
\sum_{Q\in\mathcal{W}_0} I_3(Q)&\le C(c_\mu)K_{p,p}\int_{B_0\setminus \{M^\sharp u=0\}} \bigg(   \int_0^\infty \lambda^{-\varepsilon}  \mathbf{1}_{U^{\lambda}\setminus U^{4^k\lambda}}(x)  \frac{d\lambda}{\lambda}  \bigg)g(x)^p\,d\mu(x)\\
& \le  C(c_\mu)C(k,\varepsilon) K_{p,p}\int_{B_0\setminus \{M^\sharp u=0\}}g(x)^p (M^{\sharp} u(x))^{-\varepsilon}\,d\mu(x)\,.\end{align*}
Combining the estimates above, we arrive at the desired conclusion.
\qed

\section{Keith--Zhong theorems}\label{s.Keith_Zhong}

We consider the self-improvement
properties of Poincar\'e inequalities involving {\em $p$-weak upper gradients}; in particular, we recover
the so-called Keith--Zhong Theorem \cite{MR2415381};
see Theorem \ref{t.kz}.
Moreover, we obtain a partial version of this result for
$\mathcal{D}$-structures in Theorem \ref{t.kz_D}.

\begin{definition}
Fix an exponent $1<p<\infty$.
A measurable function $g\colon X\to [0,\infty]$ is a {\em $p$-weak upper gradient (with respect to $X$)}
of a function $u\colon X\to \R$ 
if inequality
\begin{equation}\label{e.modulus}
\lvert u(\gamma(0))-u(\gamma(\ell_\gamma))\rvert \le \int_\gamma g\,ds
\end{equation}
holds for $p$-almost every curve $\gamma:[0,\ell_\gamma]\to X$; i.e., there exists a non-negative Borel function $\rho\in L^p(X)$ such that 
$\int_\gamma \rho\,ds=\infty$ whenever inequality~\eqref{e.modulus} does not hold or is not defined.
\end{definition}

We refer to~\cite{MR2867756,MR1800917,MR3363168} for further information
on $p$-weak upper gradients.

\medskip

Fix an exponent $1<p<\infty$. For each function $u\in \mathrm{Lip}(X)$, we let $\mathcal{D}^{1,p}_N(u)$ be the family of
all  $p$-weak upper gradients $g\in L^p_{\textup{loc}}(X)$ of $u$; by $g\in L^p_{\textup{loc}}(X)$ we mean that
for each $x\in X$ there exists $r_x>0$ such that $g\in L^p(B(x,r_x))$.
The properties (D2) and (D3) in Definition \ref{d.D_structure} are rather well known, see for instance \cite[Corollary 1.39]{MR2867756}. 
The property (D4) with $\beta=1$ is a consequence of a so-called `Glueing lemma', we refer to
 \cite[Lemma 2.19, Remark 2.28]{MR2867756}.
However, the $(1,p)$-Poincar\'e inequality condition (D1), with $\beta=1$, is not always valid,
 and therefore we need to assume this in some of the forthcoming results.

Beyond these properties (D1)--(D4), we also need some other  observations.
The above family $\mathcal{D}^{1,p}_N(u)$ has the
following minimality property:
if  $u\in \mathrm{Lip}(X)$, then there exists
a $p$-weak upper gradient
$g_u\in \mathcal{D}^{1,p}_N(u)$ such that $g_u\le g$ almost
everywhere if $g\in \mathcal{D}^{1,p}_N(u)$; see~\cite[Theorem 2.25]{MR2867756}.
Moreover, this {\em minimal $p$-weak upper gradient} $g_u$ is unique up to sets of measure zero in $X$.
The following result is an adaptation of \cite[Lemma 4.7]{MR1809341}; see also \cite{MR3263465}.
The proof below relies on
a localization property  of the minimal $p$-weak upper gradient to open sets (e.g.\
to balls $B_0$ in $X$). 
At this stage, the reader is encouraged to recall Definition \eqref{d.m_def}.

\begin{lemma}\label{l.Lip_estimate}
Let $1<p<\infty$. Let $B_0\subset X$ be a ball and $\mathcal{B}_0=\{B(x,r)\,:\,B(x,2r)\subset B_0\}$.
Suppose that  $u\colon X\to \R$ is a Lipschitz function and 
let $g_u\in L^p_{\textup{loc}}(X)$ be its minimal $p$-weak upper gradient.
Then inequality
\begin{equation}\label{e.desired}
g_u(x)\le C(c_\mu) M^{\sharp,p}_{1,\mathcal{B}_0}u(x)
\end{equation}
holds for almost every $x\in B_0$.
\end{lemma}

\begin{proof}
In the proof, we only consider the difficult case $B_0\not=X$;
the  case $B_0=X$ is similar.
Let $u\colon X\to \R$ be Lipschitz, with a constant $\kappa>0$. 
Write $g=C(1,c_\mu) M^{\sharp,p}_{1,\mathcal{B}_0}u$, where
the constant $C(1,c_\mu)>0$ is as in the proof of Lemma \ref{l.arm_local}.
First we  show that $4g|_{B_0}$ is a $p$-weak upper gradient of $u|_{B_0}$ with respect to $B_0$.

To begin with, we observe that $\{y\in B_0\,:\, g (y)>\lambda\}$
is an open set if $\lambda\in \R$. Hence, the
function $g|_{B_0}$ is Borel in $B_0$.
Fix a curve $\gamma\colon [0,\ell_\gamma]\to B_0\subset X$, and then fix a natural number $n\ge 2$ satisfying condition
\begin{equation}\label{e.choice}
n> 2\ell_\gamma\bigg(\frac{\dist(\gamma[0,\ell_\gamma],X\setminus B_0)}{128}\bigg)^{-1}>0\,.
\end{equation}
At the end, we will let $n$ tend to infinity.
We consider the covering $[0,\ell_\gamma]=\cup_{i=0}^{n-1} [t_i,t_{i+1}]$,
where each $t_j=j\ell_\gamma/n$. Write $\gamma_i=\gamma|_{[t_i,t_{i+1}]}$ and
$\lvert \gamma_i\rvert=\gamma[t_i,t_{i+1}]\subset B_0$ for each $i=0,\ldots,n-1$.
For each such $i$ we pick $x_i=x_i(n)\in \lvert \gamma_i\rvert$ such that
\[
g(x_i)\le \vint_{\gamma_i} g\, ds\,.
\]
Consider a fixed $i=0,\ldots,n-2$.
Note first that 
\begin{equation}\label{e.dist_i}
d(x_i,x_{i+1})\le \ell(\gamma_i)+\ell(\gamma_{i+1})=2\ell_\gamma/n=2\ell(\gamma_i)=2\ell(\gamma_{i+1})\,.
\end{equation}
Moreover, since $B_0\subsetneq X$, we can
choose a Whitney ball $Q_i\in \mathcal{W}(B_0)$
such that $x_i\in Q_i$; we refer to \S\ref{ss.Whitney}. By
using inequalities \eqref{e.choice} and \eqref{e.dist_i},  it is straightforward
to show that $x_i,x_{i+1}\in Q_i^*$. Hence, proceeding as in the proof of Lemma \ref{l.arm_local},
we see that
\[
\lvert u(x_{i})-u(x_{i+1})\rvert
\le d(x_i,x_{i+1})\big( g(x_i)+ g(x_{i+1})\big)\,.
\]
Thus, we obtain that
\begin{align*}
\lvert u(x_0)-u(x_{n-1})\rvert &\le \sum_{i=0}^{n-2} \lvert u(x_i)-u(x_{i+1})\rvert\\
&\le \sum_{i=0}^{n-2} d(x_i,x_{i+1}) \big( g(x_i)+g(x_{i+1})\big)\\
&\le 4\sum_{i=0}^{n-1} \ell(\gamma_i)\vint_{\gamma_i} g\,ds
=  4\int_\gamma g\,ds\,.
\end{align*}
By taking $n\to \infty$ and using the continuity of $u$, 
together with the facts that $x_0(n)\to \gamma(0)$ and $x_{n-1}(n)\to \gamma(\ell_\gamma)$, 
we conclude that
\[\lvert u(\gamma(0))-u(\gamma(\ell_\gamma))\rvert \le \int_\gamma 4 g \,ds\,.\]
This inequality shows  
that $4g|_{B_0}$ is a $p$-weak upper gradient of
$u|_{B_0}$ with respect to $B_0$, and
since $0\le 4g\le 4C(1,c_\mu) \kappa$, it also holds
that $4g|_{B_0}\in L^p(B_0)$. Inequality \eqref{e.desired} 
for $C(c_\mu)=4C(1,c_\mu)$
now follows from the fact that the restriction $g_u|_{B_0}$ is the minimal $p$-weak upper
gradient of $u|_{B_0}$ with respect to the (open) ball $B_0$; we refer to \cite[Lemma 2.23]{MR2867756}.
\end{proof}

The following result 
is now a consequence of our main result, Theorem \ref{t.main_local}.
This result can be further strenghtened by using Theorem \ref{t.qp_poincare},
but we leave details to the interested reader.

\begin{theorem}\label{t.kz}
Suppose that $X$ is a geodesic space and fix $1<p<\infty$.
Suppose that there are constants $K>0$ and $\tau \ge 1$ such that
the $(1,p)$-Poincar\'e inequality
\[
\vint_{B} \lvert u(x)-u_{B}\rvert\,d\mu(x)
\le K^{1/p}\mathrm{diam}(B)\bigg(\vint_{\tau B} g(x)^p \,d\mu(x)\bigg)^{1/p}
\]
holds whenever $B$ is a ball in $X$ and $g\in L^p_{\mathrm{loc}}(X)$ is a $p$-weak upper gradient of $u\in \mathrm{Lip}(X)$. 
Then
there exists a number $0<\varepsilon<p-1$ and a constant $C>0$, both of which are quantitative, such that inequality
\[
\bigg(\vint_B \lvert u(x)-u_B\rvert^p\,d\mu(x)\bigg)^{1/p}\le 
C\diam(B)\bigg(\vint_{2B} g(x)^{p-\varepsilon}\,d\mu(x)\bigg)^{1/(p-\varepsilon)}
\]
holds whenever $B\subset X$ is a ball, $u\in\mathrm{Lip}(X)$  and $g\in L^p_{\textup{loc}}(X)$ is a $p$-weak upper gradient of $u$.
\end{theorem}

\begin{proof}
From the above considerations and the standing assumptions, it follows that 
the family $\{\mathcal{D}^{1,p}_N(u)\,:\,u\in\mathrm{Lip}(X)\}$
is a $\mathcal{D}$-structure in $X$, with exponents
$p$ and $\beta=1$, and constants $K>0$ and $\tau\ge 1$.
This allows us to fix $0<\varepsilon<\varepsilon_0$ as in
Corollary \ref{c.main}.
Fix also a ball $B\subset X$, and write $B_0=2B$ and
$\mathcal{B}_0=\{B(x,r)\,:\,B(x,2r)\subset B_0\}$.
Since $B\in\mathcal{B}_0$, we have
\[
\mu(B)\bigg(\frac{1}{\diam(B)^p}\vint_B \lvert u(x)-u_B\rvert^p\,d\mu(x)\bigg)^{(p-\varepsilon)/p}
\le \int_{B_0} \big(M^{\sharp,p}_{1,\mathcal{B}_0} u\big)^{p-\varepsilon}\,d\mu\,.
\]
We  apply Corollary \ref{c.main} with the ball $B_0\subset X$ and
with the minimal $p$-weak upper gradient $g_u\in L^p_{\textup{loc}}(X)$ of $u$.
Lemma \ref{l.Lip_estimate} is needed to obtain the estimate
\[\int_{B_0\setminus \{M^{\sharp,p}_{1,\mathcal{B}_0} u=0\}} g_u^p\big( M^{\sharp,p}_{1,\mathcal{B}_0} u\big)^{-\varepsilon}\,d\mu
\le C(c_\mu,\varepsilon)
\int_{B_0} g_u^{p-\varepsilon}\,d\mu
\]
for the right-hand side of \eqref{e.loc_des_C}. At the end we use the fact that $g_u^{p-\varepsilon}\le g^{p-\varepsilon}$ almost everywhere if $g\in L^p_{\textup{loc}}(X)$ is
any $p$-weak upper gradient of $u$.
\end{proof} 

Along the same lines, we can also prove a 
version of the Keith--Zhong theorem for general $\mathcal{D}$-structures
in geodesic spaces.
This result is stated in Theorem \ref{t.kz_D} below.
This result is a true generalization of Theorem \ref{t.kz} 
but it is unknown to the authors  whether the additional minimality
condition in the statement below can be removed.

\begin{theorem}\label{t.kz_D}
Suppose that we are given a $\mathcal{D}$-structure
in a geodesic space $X$, with exponents 
$1< p<\infty$ and $0<\beta \le 1$, and constants $K>0$ and $\tau \ge 1$.
Fix  $\eta >0$.
Then
there exists a number $0<\varepsilon<p-1$ 
and a constant $C>0$, both of which are quantitative, such that inequality
\[
\bigg(\vint_B \lvert u(x)-u_B\rvert^p\,d\mu(x)\bigg)^{1/p}\le 
C\diam(B)^{\beta}\bigg(\vint_{2B} g(x)^{p-\varepsilon}\,d\mu(x)\bigg)^{1/(p-\varepsilon)}
\]
holds whenever $B\subset X$ is a ball, $u\in\mathrm{Lip}_\beta(X)$ and $g\in \mathcal{D}(u)$
satisfies the following minimality condition: $g \le \eta M^{\sharp,p}_{\beta,\mathcal{B}_0} u$
almost everywhere in $B_0=2B$.
\end{theorem}

Theorem \ref{t.kz_D} can also be further strenghtened by using Theorem \ref{t.qp_poincare},
but again we leave details to the interested reader.

\section{Axiomatic Sobolev spaces}\label{s.axiomatic}

A given $\mathcal{D}$-structure gives
rise to a  Sobolev space;  cf.\  \cite{MR1876253}.
Our main result in this section is a certain norm-equivalence 
for such spaces,  Theorem \ref{t.structure}.

\subsection{Sobolev spaces and $\mathcal{D}$-structures}

Let us begin with the definition of an abstract Sobolev space that is
defined in terms of a $\mathcal{D}$-structure.
Our treatment is inspired by \cite{MR1876253}.
See also \cite{MR2039955,MR2508848}
for further references on this type of abstract Sobolev spaces.

\begin{definition}
Given a $\mathcal{D}$-structure $\mathcal{D}$ in $X$, with exponents
$1\le p<\infty$ and $0<\beta\le 1$, the associated Sobolev space $W^{p}_\beta(X,\mathcal{D})$
is the completion \cite{MR992618} of the vector space
\[
\{u\in \mathrm{Lip}_\beta(X)\,:\, \lVert u\rVert_{W^{p}_\beta(X,\mathcal{D})}<\infty\}
\]
that is equipped with the norm\footnote{Conditions (D2) and (D3) imply the properties of a vector space and a norm; cf.\ \cite[Theorem 1.5]{MR1876253}.}
\[
\lVert u\rVert_{W^{p}_\beta(X,\mathcal{D})} = \Bigl( \lVert u\rVert_{L^p(X)}^p+
\inf_{g\in\mathcal{D}(u)} \lVert g\rVert_{L^p(X)}^p\Bigr)^{1/p}\,,\qquad u\in\mathrm{Lip}_\beta(X)\,.
\]
\end{definition}

In order to formulate our results, we need a global version of the maximal function \eqref{d.m_def}.
To this end, we write $\mathcal{B}=\{B\,:\,B\subset X\text{ is a ball}\,\}$
and denote 
\[M^{\sharp,p}_\beta u(x) = M^{\sharp,p}_{\beta,\mathcal{B}} u(x)
=\sup_{x\in B\in\mathcal{B}} \bigg(\frac{1}{\diam(B)^{\beta p}}\vint_B \lvert u(y)-u_B\rvert^p\,d\mu(y)\bigg)^{1/p}\,,\qquad x\in X\,,
\]
whenever $u\colon X\to \R$ is a $\beta$-H\"older function, i.e., $u\in \mathrm{Lip}_\beta(X)$. 

By applying this global maximal function in Theorem \ref{t.structure} below, we provide a structure independent representation
for the Sobolev norm that arises from an appropriate $\mathcal{D}$-structure;
more specifically, we need to additionally assume that 
$\eta M^{\sharp,p}_\beta u\in \mathcal{D}(u)$
whenever $u$ is a $\beta$-H\"older function on $X$.
Here $\eta$ is a constant that is independent of $u$.
 As we will
see, this assumption holds in various applications.

\begin{theorem}\label{t.structure}
Suppose we are given  a $\mathcal{D}$-structure $\mathcal{D}$  in a geodesic space $X$, with
exponents $1<p<\infty$ and $0<\beta\le 1$. Let $\eta>0$ and suppose 
that $\eta M^{\sharp,p}_\beta u\in\mathcal{D}(u)$ for every
$u\in \mathrm{Lip}_\beta(X)$. Then there exists a constant $C=C(K_{p,p},\beta,p,c_\mu,\eta)\ge 1$ such that
\begin{equation}\label{e.norm_equiv_I}
C^{-1}\lVert u\rVert_{W^{p}_\beta(X,\mathcal{D})}\le \big(\lVert u\rVert_{L^p(X)}^p + 
\lVert M^{\sharp,p}_\beta u\rVert_{L^p(X)}^p\big)^{1/p}
\le C\lVert u\rVert_{W^{p}_\beta(X,\mathcal{D})}
\end{equation}
whenever $u\in \mathrm{Lip}_\beta(X)$. 
\end{theorem}

\begin{remark}
Inequality \eqref{e.norm_equiv_I} holds also when either one of the two quantities 
\[
\lVert u\rVert_{W^{p}_\beta(X,\mathcal{D})}\,,\qquad \big(\lVert u\rVert_{L^p(X)}^p + 
\lVert M^{\sharp,p}_\beta u\rVert_{L^p(X)}^p\big)^{1/p}
\]
is infinite. In this case we can conclude that actually both of these quantities are infinite.
\end{remark}

\begin{proof}[Proof of Theorem \ref{t.structure}]
The left inequality in~\eqref{e.norm_equiv_I} follows from the definitions and the assumption
that $\eta M^{\sharp,p}_\beta u\in\mathcal{D}(u)$ for every
$u\in \mathrm{Lip}_\beta(X)$. 
To prove the right inequality, we fix a point $x_0\in X$ and denote $B_j=B(x_0,j)$ and
 $\mathcal{B}_j=\{B=B(x,r)\,:\,2B\subset B_j\}$ for $j\in \N$.

Fix $u\in \Lip_\beta(X)$ and $g\in \mathcal{D}(u)$.
Observe that
\[
M^{\sharp,p}_\beta u(x)= \lim_{j\to \infty} \big(\mathbf{1}_{B_j}(x) M^{\sharp,p}_{\beta,\mathcal{B}_j} u(x)\big)
\]
whenever $x\in X$.
Hence, by Fatou's lemma and Theorem \ref{t.main_local}, with $\varepsilon=0$, we obtain  that
\begin{align*}
\int_{X}\big( M^{\sharp,p}_{\beta} u(x)\big)^{p}\,d\mu(x)
&\le \liminf_{j\to\infty} \int_{B_j} \big( M^{\sharp,p}_{\beta,\mathcal{B}_j} u(x)\big)^{p}\,d\mu(x)
\\&\le C(K_{p,p},\beta,p,c_\mu) \liminf_{j\to\infty}\int_{B_j}  g(x)^p\,d\mu(x)
\\
&\le C(K_{p,p},\beta,p,c_\mu) \int_X g(x)^p\,d\mu(x)\,.
\end{align*}
The right inequality in \eqref{e.norm_equiv_I}
follows by infimizing the above estimate over all $g\in\mathcal{D}(u)$.
\end{proof}

\subsection{Universality of Haj{\l}asz--Sobolev spaces}
By using Theorem \ref{t.structure}, we shall now provide
isomorphic representatives for the abstract Sobolev spaces
in terms of certain Haj{\l}asz--Sobolev spaces. 
This can be done as follows if $1<p<\infty$ and $0<\beta\le 1$.

For each $\beta$-H\"older function $u\colon X\to \R$, we let 
$\mathcal{D}_H^{\beta,p}(u)\not=\emptyset$ 
be the family of all measurable functions $g\colon X\to [0,\infty]$ such that
\begin{equation}\label{e.hajlasz}
\lvert u(x)-u(y)\rvert \le d(x,y)^\beta\big( g(x)+g(y) \big)
\end{equation}
almost everywhere, i.e., there exists an exceptional set $N=N(g)\subset X$ for which $\mu(N)=0$ and
inequality \eqref{e.hajlasz} holds for every $x,y\in X\setminus N$. 

As we will see below, this construction gives  a $\mathcal{D}$-structure
\[
\mathcal{D}_{H}^{\beta,p}=\{\mathcal{D}_{H}^{\beta,p}(u)\,:\, u\in \Lip_\beta(X)\},
\] 
with exponents $p$ and $\beta$, and with constants $K=2^p$ and $\tau=1$.
The associated abstract Sobolev space is the so-called {\em Hajlasz--Sobolev
space} that is denoted by \[M^{\beta,p}(X)=W^{p}_\beta(X,\mathcal{D}_{H}^{\beta,p})\,.\] This space
has been studied, e.g., in  \cite{MR2039955,MR1681586,MR1800917}.
Our approach via completion is 
not standard. However, by the
known density results of H\"older-functions \cite[Proposition 4.5]{MR3357989}, the perhaps more conventional definition 
\cite[pp.~194--195]{MR3357989}
yields an isomorphic Banach space.

Returning to the  $\mathcal{D}$-structure conditions, it is straightforward to verify that condition (D1) is valid; in fact, even the stronger $(p,p)$-Poincar\'e inequality condition (D1')
in Theorem \ref{t.pp_poincare}
holds with a constant  $K_{p,p}=2^p$; cf.\ \cite[Theorem 5.15]{MR1800917}. The two conditions (D2) and (D3) are also satisfied; we leave details to the reader. The
validity of the last condition (D4) is a consequence of 
the following lemma.

\begin{lemma}\label{l.leibniz}
Let $1<p<\infty$ and $0<\beta\le 1$, and fix a Borel set $E\subset X$.
Let $u\colon X\to \R$ be a $\beta$-H\"older function and
suppose that $v\colon X\to \R$ 
is such that $v|_{X\setminus E} =u|_{X\setminus E}$ and
there exists a constant $\kappa\ge 0$ such that
$\lvert v(x)-v(y)\rvert \le \kappa\, d(x,y)^\beta$
for all $x,y\in X$.
 Then
\[
g_v=\kappa\, \mathbf{1}_{E} + g_u\mathbf{1}_{X\setminus E} \in \mathcal{D}_{H}^{\beta,p}(v)
\]
whenever $g_u\in \mathcal{D}_H^{\beta,p}(u)$.
\end{lemma}

\begin{proof}
Fix a function $g_u\in \mathcal{D}_H^{\beta,p}(u)$ and let 
$N\subset X$ be the exceptional set such that $\mu(N)=0$ and 
inequality \eqref{e.hajlasz} holds for every $x,y\in X\setminus N$ and with $g=g_u$.

Fix $x,y\in X\setminus N$. If $x,y\in X\setminus E$, then
\[
\lvert v (x)-v(y)\rvert =\lvert u (x)-u(y)\rvert 
\le d(x,y)^\beta \big(g_u(x)+g_u(y)\big) = d(x,y)^\beta \big(g_v (x)+g_v (y)\big)\,.
\]
If $x\in E$ or $y\in E$, then
\[
\lvert v (x)-v(y)\rvert \le \kappa\, d(x,y)^\beta\le d(x,y)^\beta\big(g_v (x)+g_v (y)\big)\,.
\]
By combining the estimates above, we find that
\[
\lvert v (x)-v(y)\rvert \le d(x,y)^\beta\big(g_v (x)+g_v (y)\big)
\]
whenever $x,y\in X\setminus N$. The desired conclusion $g_v\in\mathcal{D}_H^{\beta,p}(v)$ follows.
\end{proof}

The following corollary is a universality 
result for  Haj{\l}asz--Sobolev spaces $M^{\beta,p}(X)$. Namely,  
any abstract Sobolev space, rising from a suitable $\mathcal{D}$-structure, turns to be isomorphic to this particular Sobolev space $M^{\beta,p}(X)$.

\begin{corollary}\label{c.universal}
Suppose we are given  a $\mathcal{D}$-structure $\mathcal{D}_A$ in a geodesic space $X$, with
exponents $1<p<\infty$ and $0<\beta\le 1$.
Assume that there exists $\eta>0$ such 
that $\eta M^{\sharp,p}_\beta u\in\mathcal{D}_A(u)$ for every
$u\in \mathrm{Lip}_\beta(X)$.
Let $W^{\beta}_p(X)=W^{\beta}_p(X,\mathcal{D}_A)$.
 Then there exists
$C\ge 1$ such that
\begin{equation}\label{e.norm_equiv}
C^{-1}\lVert u\rVert_{W^p_\beta(X)}\le \lVert u\rVert_{M^{\beta,p}(X)}
\le C\lVert u\rVert_{W^p_\beta(X)}
\end{equation}
whenever $u\in \mathrm{Lip}_\beta(X)$.
Moreover, there exists a unique Banach-space isomorphism between the spaces $W^{p}_\beta(X)$ and $M^{\beta,p}(X)$ which is the identity on 
$W^{p}_\beta(X)\cap \mathrm{Lip}_\beta(X)=M^{\beta,p}(X)\cap \mathrm{Lip}_\beta(X)$.
\end{corollary}

\begin{proof}
By modifying the proof of inequality \eqref{e.des_hld}, 
it is straightforward to show that
\[C(\beta,c_\mu) M^{\sharp,p}_\beta u\in\mathcal{D}_{H}^{\beta,p}(u)\,,\qquad u\in\mathrm{Lip}_\beta(X)\,.\]
Applying Theorem \ref{t.structure} with the two 
$\mathcal{D}$-structures 
$\mathcal{D}_A$ and $\mathcal{D}_{H}^{\beta,p}$
yields the claim. Indeed, recall
that by definitions $W^{\beta}_p(X)=W^{\beta}_p(X,\mathcal{D}_A)$ and $M^{\beta,p}(X)=W^{p}_\beta(X,\mathcal{D}^{\beta,p}_H)$.
\end{proof}

\subsection{Universality of Newtonian spaces}
Fix $u\in \mathrm{Lip}(X)$ and $1<p<\infty$.
Recall
from \S\ref{s.Keith_Zhong} that  $\mathcal{D}^{1,p}_N(u)$ is
the family of all  $p$-weak upper gradients $g\in L^p_{\textup{loc}}(X)$ of the function $u$. 
Arguing as in  \S\ref{s.Keith_Zhong}, we find that 
\[
\mathcal{D}^{1,p}_N=\{\mathcal{D}^{1,p}_N(u)\,:\,u\in\Lip(X)\}
\] 
is a $\mathcal{D}$-structure with exponents $\beta=1$ and $p$, if we {\em assume} that the $(1,p)$-Poincar\'e  inequality condition (D1) holds.
The associated abstract Sobolev space
is the so-called {\em Newtonian space} 
\[N^{1,p}(X)=W^p_1(X,\mathcal{D}^{1,p}_N)\,.\]
We remark that this notation is not entirely standard
since Lipschitz functions need not be dense when the more conventional approach 
\cite{MR1809341}, \cite[Definition 1.17]{MR2867756} to the Newtonian space is adopted.
However, when $X$
supports a $(1,p)$-Poincar\'e inequality \eqref{e.1p}
for all measurable functions instead of Lipschitz functions only, then the density
result holds and the two definitions give
isomorphic Banach spaces; c.f. \cite[Theorem 5.1]{MR2867756}.

The following corollary   extends and complements   \cite[Theorem 4.9]{MR1809341} and \cite[Theorem 4.3]{MR2508848}. Observe that, by Corollary \ref{c.universal} and transitivity, it also produces a universality result for
Newtonian spaces $N^{1,p}(X)$.

\begin{corollary}
Suppose that $X$ is a geodesic space and fix $1<p<\infty$.
Suppose that there are constants $K>0$ and $\tau \ge 1$ such that
the $(1,p)$-Poincar\'e inequality
\begin{equation}\label{e.1p}
\vint_{B} \lvert u(x)-u_{B}\rvert\,d\mu(x)
\le K^{1/p}\mathrm{diam}(B)\bigg(\vint_{\tau B} g(x)^p \,d\mu(x)\bigg)^{1/p}
\end{equation}
holds whenever $B$ is a ball in $X$ and $g\in L^p_{\mathrm{loc}}(X)$ is a $p$-weak upper gradient of $u\in \mathrm{Lip}(X)$. 
Then there exists a constant $C\ge 1$ such that
\[ 
C^{-1}\lVert u\rVert_{N^{1,p}(X)}\le \lVert u\rVert_{M^{1,p}(X)}
\le C\lVert u\rVert_{N^{1,p}(X)}
\] 
whenever $u\in \mathrm{Lip}(X)$.
Moreover, there exists a unique Banach-space isomorphism between the spaces $M^{1,p}(X)$ and $N^{1,p}(X)$ which is the identity on $N^{1,p}(X)\cap \mathrm{Lip}(X)=M^{1,p}(X)\cap \mathrm{Lip}(X)$.
\end{corollary}

\begin{proof}
Arguing as in \S\ref{s.Keith_Zhong}, we obtain
a constant $C(1,c_\mu)>0$ such that
$4C(1,c_\mu)M^{\sharp,p}_1 u\in L^p_{\textup{loc}}(X)$ is a $p$-weak upper gradient
of any given $u\in\mathrm{Lip}(X)$.
The claim follows from Corollary \ref{c.universal}.
\end{proof}

\bibliographystyle{abbrv}

\begin{thebibliography}{10}

\bibitem{MR2867756}
A.~Bj{\"o}rn and J.~Bj{\"o}rn.
\newblock {\em Nonlinear potential theory on metric spaces}, volume~17 of {\em
  EMS Tracts in Mathematics}.
\newblock European Mathematical Society (EMS), Z\"urich, 2011.

\bibitem{EB2016}
S.~Eriksson-Bique.
\newblock Alternative proof of Keith-Zhong self-improvement.
\newblock {\em arXiv:}1610.02129, 2016.

\bibitem{MR807149}
J.~Garc{\'{\i}}a-Cuerva and J.~L. Rubio~de Francia.
\newblock {\em Weighted norm inequalities and related topics}, volume 116 of
  {\em North-Holland Mathematics Studies}.
\newblock North-Holland Publishing Co., Amsterdam, 1985.

\bibitem{MR3089750}
A.~Gogatishvili, P.~Koskela, and Y.~Zhou.
\newblock Characterizations of {B}esov and {T}riebel-{L}izorkin spaces on
  metric measure spaces.
\newblock {\em Forum Math.}, 25(4):787--819, 2013.

\bibitem{MR1876253}
V.~Gol{\cprime}dshtein and M.~Troyanov.
\newblock Axiomatic theory of {S}obolev spaces.
\newblock {\em Expo. Math.}, 19(4):289--336, 2001.

\bibitem{MR2039955}
P.~Haj\l asz.
\newblock Sobolev spaces on metric-measure spaces.
\newblock In {\em Heat kernels and analysis on manifolds, graphs, and metric
  spaces ({P}aris, 2002)}, volume 338 of {\em Contemp. Math.}, pages 173--218.
  Amer. Math. Soc., Providence, RI, 2003.

\bibitem{MR1681586}
P.~Haj{\l}asz and J.~Kinnunen.
\newblock H\"older quasicontinuity of {S}obolev functions on metric spaces.
\newblock {\em Rev. Mat. Iberoamericana}, 14(3):601--622, 1998.

\bibitem{MR1336257}
P.~Haj{\l}asz and P.~Koskela.
\newblock Sobolev meets {P}oincar\'e.
\newblock {\em C. R. Acad. Sci. Paris S\'er. I Math.}, 320(10):1211--1215,
  1995.

\bibitem{MR1800917}
J.~Heinonen.
\newblock {\em Lectures on analysis on metric spaces}.
\newblock Universitext. Springer-Verlag, New York, 2001.

\bibitem{MR3363168}
J.~Heinonen, P.~Koskela, N.~Shanmugalingam, and J.~T. Tyson.
\newblock {\em Sobolev spaces on metric measure spaces: An approach based on upper gradients}, volume~27 of {\em New
  Mathematical Monographs}.
\newblock Cambridge University Press, Cambridge, 2015.

\bibitem{MR3263465}
R.~Jiang, N.~Shanmugalingam, D.~Yang, and W.~Yuan.
\newblock Haj\l asz gradients are upper gradients.
\newblock {\em J. Math. Anal. Appl.}, 422(1):397--407, 2015.

\bibitem{MR2415381}
S.~Keith and X.~Zhong.
\newblock The {P}oincar\'e inequality is an open ended condition.
\newblock {\em Ann. of Math. (2)}, 167(2):575--599, 2008.

\bibitem{MR992618}
E.~Kreyszig.
\newblock {\em Introductory functional analysis with applications}.
\newblock Wiley Classics Library. John Wiley \& Sons, Inc., New York, 1989.

\bibitem{MR1809341}
N.~Shanmugalingam.
\newblock Newtonian spaces: an extension of {S}obolev spaces to metric measure
  spaces.
\newblock {\em Rev. Mat. Iberoamericana}, 16(2):243--279, 2000.

\bibitem{MR2508848}
N.~Shanmugalingam.
\newblock A universality property of {S}obolev spaces in metric measure spaces.
\newblock In {\em Sobolev spaces in mathematics. {I}}, volume~8 of {\em Int.
  Math. Ser. (N. Y.)}, pages 345--359. Springer, New York, 2009.

\bibitem{MR3357989}
N.~Shanmugalingam, D.~Yang, and W.~Yuan.
\newblock Newton-{B}esov spaces and {N}ewton-{T}riebel-{L}izorkin spaces on
  metric measure spaces.
\newblock {\em Positivity}, 19(2):177--220, 2015.

\end{thebibliography}
\def\cprime{$'$} \def\cprime{$'$} \def\cprime{$'$}

\end{document}